\documentclass[11pt]{article}

\usepackage[english]{babel}
\usepackage{amsmath,amssymb,amsfonts,amsthm,mathrsfs,bbm}
\usepackage{graphicx,xcolor}
\usepackage{enumerate,geometry,marginnote,afterpage}
\usepackage{fancyhdr,courier,type1cm,dsfont,float,tikz}
\usepackage{hyperref}
\usepackage{verbatim}
\usepackage[normalem]{ulem}
\usetikzlibrary{arrows, positioning}

\numberwithin{equation}{section}

\theoremstyle{plain} 

\newtheorem{theorem}{Theorem}[section]
\newtheorem{corollary}[theorem]{Corollary}
\newtheorem{lemma}[theorem]{Lemma}

\newcommand{\ee}{\mathrm{e}}
\newcommand{\N}{\mathbb{N}}
\newcommand{\R}{\mathbb{R}}
\newcommand{\PP}{\mathbb{P}}
\newcommand{\EE}{\mathbb{E}}



\date{\today}

\begin{document}

\title{Renormalisation of Inhomogeneous Random Graphs}

\author{
\renewcommand{\thefootnote}{\arabic{footnote}}
Luca Avena
\footnotemark[1]
\\
\renewcommand{\thefootnote}{\arabic{footnote}}
Diego Garlaschelli
\footnotemark[2]
\\
\renewcommand{\thefootnote}{\arabic{footnote}}
Rajat Subhra Hazra
\footnotemark[3]
\\
\renewcommand{\thefootnote}{\arabic{footnote}}
Frank den Hollander
\footnotemark[3]
}

\footnotetext[1]{
Dipartimento di Matematica e Informatica `Ulisse Dini', Universit\`a degli Studi di Firenze, Viale Giovanni Battista Morgagni, 67/a, 50134 Firenze, Italy.\\ 
\texttt{luca.avena@unifi.it}
}

\footnotetext[2]{
IMT Scuola di Alti Studi, Lucca, Italy \& Lorentz Institute for Theoretical Physics, Leiden University, The Netherlands.\\ 
\texttt{diego.garlaschelli@imtlucca.it}
}

\footnotetext[3]{
Mathematical Institute, Leiden University, Einsteinweg 55, 2333 CC Leiden, The Netherlands.\\
\texttt{\{r.s.hazra,denholla\}@math.leidenuniv.nl}
}

\maketitle

\begin{abstract}
We consider inhomogeneous random graphs in which vertices are assigned i.i.d.\ random weights, pairs of distinct vertices are connected by an edge independently with a probability that is a bi-variate function of the weights of the vertices, and single vertices are connected to themselves by a self-loop independently with a probability that is a uni-variate function of the weight of the vertex. We apply a renormalisation transformation in which vertices are aggregated into groups of equal size according to a greedy algorithm, namely, distinct groups of aggregated vertices are connected by an aggregated edge if and only if there is at least one edge connecting two constituent vertices across the groups, while a group of aggregated vertices is connected to itself by an aggregated self-loop if and only if there is at least one self-loop at an internal vertex or one edge connecting a pair of distinct internal vertices. We analyse what happens when the renormalisation transformation is iterated. In particular, we show that, starting from appropriately scaled connection functions, the iterated renormalised graphs converge to a two-parameter family of random graphs, acting as an attractor in a universality class. We consider a light-tailed regime, for which the scaling limit is a homogeneous Erd\H{o}s--R\'enyi random graph, and a heavy-tailed regime, for which the scaling limit is an inhomogeneous random graph with stable infinite-mean random weights and an exponential disconnection function. Different scalings are needed for the two regimes. Which of the two regimes prevails depends on the choice of the connection functions and the choice of the law of the random weights.   



\medskip\noindent
{\it Acknowledgment.} 
LA \& RSH \& FdH were supported through NWO Gravitation Grant NETWORKS-024.002.003. DG was supported through the project \emph{Network renormalization: from theoretical physics to the resilience of societies}, file number NWA.1418.24.029 of the research programme NWA L3 -- Innovative projects within routes 2024, and the project \emph{Redefining renormalization for complex networks}, file number OCENW.M.24.039, both of which are partly financed by NWO (see \url{https://doi.org/10.61686/AOIJP05368} and \url{https://doi.org/10.61686/PBSEC42210}).
\end{abstract}


\newpage


\section{Introduction and main results}

The goal of this paper is to carry out a \emph{renormalisation analysis} of \emph{inhomogeneous random graphs with random weights}. The renormalisation map considered in this paper acts on the vertices and the weight-dependent connection probabilities in the form of an \emph{aggregation algorithm}. The target is to identify \emph{universal attractors} for the iterates of the renormalisation map. Section~\ref{ss.bm} provides background and motivation. Section~\ref{ss.model} describes the model. Section~\ref{ss.map} identifies the renormalisation map. Section~\ref{ss.theorems} states the main theorems. Section~\ref{ss.discussion} offers a discussion of the main theorems and lists open problems. Section~\ref{ss.outline} gives an outline of the remainder of the paper, which provides technical definitions and proofs of key statements.


\subsection{Background and motivation}
\label{ss.bm}


\paragraph{Renormalisation.}

Coarse-graining of a model of \emph{interacting particle systems} is a way to aggregate information about the structure and the evolution of the system. With the help of \emph{renormalisation transformations} that sum out over variables on a microscopic scale while appropriately rescaling the model parameters, information is obtained about the system on a mesoscopic scale. Iteration of the transformation defines a \emph{renormalisation flow} in the space of parameters and ultimately leads to a description of the system on a macroscopic scale \cite{M,W}. 

Renormalisation plays an important role in the analysis of phase transitions and critical phenomena in equilibrium statistical physics \cite{PV}. Since parsimonious microscopic models of physical systems are ultimately based on \emph{effective} rather than \emph{fundamental} degrees of freedom, renormalisation theory aims to ensure that the macroscopic properties predicted by such models do not depend sensitively on microscopic details. In this setting, identifying possible \emph{attractors} of microscopic models under the renormalisation flow is crucial: microscopic differences between models that flow to the same attractor are macroscopically irrelevant. 

Attractors are typically identified as \emph{fixed points}, i.e., appropriately defined invariants, of the renormalisation transformation \cite{M,W,PV}. Under the action of this transformation, it is essential that the flow of parameters keeps tracks of the system within an appropriate class of target models. Generally, it is hard to work out the details, and consequently many arguments in the literature are approximative and mathematically non-rigorous. For instance, renormalisation can throw the system \emph{out} of the class of target models, while approximations presume that nonetheless the system remains close to this class in some sense \cite{vEFS}. Properly ensuring that fixed points also act as attractors along the renormalisation flow is therefore a mathematically delicate task.


\paragraph{Renormalisation of networks.}
 
Renormalisation of \emph{complex networks} \cite{GGPS}, especially in the presence of non-homogeneities and randomness, represents an open challenge and remains largely unexplored from a mathematical perspective. As for physical systems, in many empirical networks there is no unique natural resolution for the constituent vertices: for instance, individuals may be grouped into households, cities or countries, web pages into domains, neurons into regions of the brain, and firms into industries or sectors. How the representation of large-scale social, technological, biological and economic networks transforms under a change of resolution, and what microscopic network properties are irrelevant for the behaviour of the system at coarser scales, including the processes that run on it, is therefore an important and fascinating question.

Most of the current approaches to network renormalisation are concerned with the identification of meaningful iterative coarse-graining rules that, given a single input graph, produce an aggregated graph that is maximally informative about the original graph under a chosen criterion. For instance, \emph{geometric renormalisation} \cite{GPBS} proceeds by first finding an optimal embedding of the vertices of the input graph onto some metric space, and then aggregating groups of vertices that are found within a given mutual distance from one another. Edges between constituent vertices are aggregated onto edges between aggregated vertices, and the geometric embedding of the reduced graph is informative about the embedding of the original graph. By contrast, \emph{Laplacian renormalisation} \cite{VGCG} focusses on the Laplacian matrix of the original graph, which acts as the generator of a discrete diffusion process, and proceeds by averaging out the `fast eigenmodes', which leads to a reduced Laplacian matrix that mimics the same diffusion process as if it were running on an aggregated network. In \cite{ACGM}, a Laplacian-based renormalisation scheme is introduced that is similar in spirit and uses Markov intertwining dualities to construct multi-resolution schemes and wavelet transforms on finite undirected graphs.


\paragraph{Random graph models.}

A different approach, which is our main interest here, focuses on how \emph{random graph models}, along with their underlying parameters, transform under network aggregation. This approach is known as \emph{ensemble} or \emph{multiscale} network renormalisation \cite{GLG}. The basic idea is that, given a random graph model on an initial set of graphs, together with a deterministic aggregation rule projecting each graph in the initial set onto a reduced graph in some output set, an aggregated random graph model on the output set is induced. In general, the aggregated graph may have a difficult law, with parameters that are complicated functions of the parameters of the original law. Multiscale network renormalisation looks for \emph{invariant random graph ensembles} under a given aggregation rule, i.e., random graphs that remain governed by the original law of the graph, only up to a rescaling of its defining parameters. Such invariant ensembles represent the graph analogue of stable random variables.

If restricted to the class of random graphs with independent edges, the search for an invariant ensemble under a given coarse-graining rule is relatively straightforward and leads to an inhomogeneous random graph model where vertices are assigned positive real weights and a connection between two vertices is established with a probability that is a specific function of the weights of the two vertices \cite{GLG}. Under a greedy coarse-graining procedure, whereby an aggregate vertex is assigned a weight equal to the sum of the weights of the constituent vertices while two aggregate vertices are connected whenever at least one microscopic connection exists between their constituents, the invariant probability of \emph{disconnection} between two vertices can be identified as a decreasing exponential of the product of the weights of the two vertices \cite{GLG} (in the case of directed graphs, appropriate generalisations can be made \cite{LG}). This specific connection function has been considered previously in the mathematical literature \cite{Aldous1997,NorrosReittu2006,Bollobas2007,vdHofstad2017,CaronFox2017}, albeit with different purposes. 
The identification of this connection function with an invariant under a renormalisation transformation opens up new avenues.


\paragraph{Flow of the law of the weights.}

While most of the prior research focussed on the case where the vertex weights are either deterministic variables (representing e.g.\ some empirically observable amount of activity, mass, fitness, population size or capacity) or random variables drawn from convenient probability distributions, the network renormalisation perspective brings in the notion of the \emph{flow of the law of the weights} \cite{GLG}. In particular, just like the connection probability of an inhomogeneous random graph model transforms under coarse-graining, also the law of the weights transforms, and is given by the convolution of the laws of the weights to produce the weights of aggregated vertices. The identification of a \emph{joint fixed point}, i.e., the combination of an invariant connection function along with an invariant law of the weights, therefore becomes relevant. If vertices are aggregated into groups of equal size, then an invariant law is necessarily either a delta-like distribution (with vertices having equal weights at any aggregation level) or a one-sided stable-law distribution (with vertices having weights distributed with asymptotic power-law decay).

The \emph{former} fixed point corresponds to the homogeneous Erd\H{o}s--R\'enyi random graph where, irrespectively of the form of the connection function, the value of the connection probability is the same for all pairs of vertices (the value itself flowing appropriately under renormalisation). Instead, the \emph{latter} fixed point corresponds, in the regime of infinite-mean fitness, to an extremely inhomogeneous invariant model, also called the Multi Scale Model (MSM) \cite{GLG}, which is different from any other inhomogeneous random graph model with independent edges and finite-mean weights \cite{vdHofstad2017}. In particular, previous mathematical investigations of the MSM have highlighted that, under an appropriate scaling of a global parameter tuning the overall edge density, the degree distribution has a universal asymptotic inverse-square decay \cite{AGHL}, the local clustering coefficient remains finite even in the sparse regime where the edge density vanishes \cite{CvdHG}, the fraction of isolated vertices and other non-self-averaging network properties converge to random variables \cite{CvdHG}, and the spectral properties of the adjacency matrix are quite exotic, with multiple largest eigenvalues scaling like the square root of the number of vertices and the corresponding eigenvector entries being characterised by log-periodicity and complex scaling exponents \cite{CHG}.


\paragraph{Attractors.}

Motivated by the above results, in the present paper we address for the first time the natural question whether the aforementioned invariant graph ensembles are also attractors of the renormalisation flow for generic random graph models. In particular, we ask whether, starting from arbitrary inhomogeneous random graphs with independent edges, random weights, and weight-dependent connection probabilities, the iteration of the greedy aggregation rule progressively transforms the original model into one of the known fixed points. Note that the coarse-grained network records the existence of a connection between any two aggregate vertices, not the identity of the microscopic connections between the constituent vertices themselves. It therefore naturally represents the original system at a lower resolution level. As illustrated in the remainder of this section and proven rigorously in the subsequent sections, we establish conditions (referred to as light-tailed and heavy-tailed regimes, respectively) under which a random graph model is attracted either to the Erd\H{o}s--R\'enyi random graph or to the invariant model with one-sided stable-law weights and exponential disconnection probability. Although we are able to only partially identify the basins of attraction of the attractors, our results highlight that the iteration of a fundamental coarse-graining procedure (representing a change in the resolution at which random graphs are observed) can naturally bring the network close to one of two very different models, characterised by full homogeneity or extreme heterogeneity, respectively.


\subsection{Model}
\label{ss.model}

The object of interest in this paper is $G=(V,E)$, the \emph{inhomogeneous random graph} with vertex set $V = [n] = \{1,\ldots,n\}$ and edge set $E \subset [n] \times [n]$ generated as follows \cite{GLG}, \cite{GGPS}. Let $X=\{X_i\}_{i \in [n]}$ be i.i.d.\ weights taking values in $\R_+ = [0,\infty)$ drawn from a common probability distribution. Choose functions
\[
f^{\neq} \colon\,\R_+ \times \R_+ \to [0,1),\qquad f^{=}\colon\,\R_+ \to [0,1),
\] 
called \emph{connection functions}. Connect the unordered pair of vertices $\{i,j\}$, $i \neq j$, with probability $f^{\neq}_{ij}[X] = f^{\neq}(X_i,X_j)$ independently for different pairs of vertices, and connect the vertex $i$ with itself with probability $f^{=}_{i}[X] = f^{=}(X_i)$ independently for different vertices. Note that we allow for single edges (when $i \neq j$) and single self-loops (when $i=j$). We do \emph{not allow for multiple edges or multiple self-loops}. Because edges are undirected, we require $f^{\neq}$ to be symmetric. 

Note that $G$ has \emph{two types of randomness}, one coming from $X$ associated with the vertices and one coming from the edges drawn according to $f^{\neq}, f^{=}$ \emph{conditional} on $X$. Sometimes we write $G[X]$ to exhibit that $G$ depends on $X$. Throughout the sequel, $\EE_X$ denotes expectation with respect to $X$.


\subsection{Renormalisation}
\label{ss.map}

Suppose that $n=m^L$ for some $m\in\N\setminus\{1\}$ and $L \in \N$. Partition the $n$ vertices in $G$ into disjoint groups of $m$ vertices each, and \emph{aggregate} the vertices in each group into a single vertex, referred to as a level-$1$ vertex. Construct an \emph{aggregated graph} $G^{(1)}$ consisting of $n_1\equiv n/m$ level-$1$ vertices, labelled by $[n/m]$, by drawing an aggregated edge between two level-$1$ vertices $I^{(1)}$ and $J^{(1)}$ \emph{if and only if} there is a pair of vertices $i \in I^{(1)}$ and $j \in J^{(1)}$ in $G$ that is connected by an edge in $G$. For $I^{(1)}=J^{(1)}$ we allow $i=j$, so that an aggregated self-loop at $I^{(1)}$ is drawn \emph{if and only if} either an edge between a pair of distinct constituent vertices $i \ne j$ in $I^{(1)}$ is present or a self-loop at any constituent vertex $i \in I^{(1)}$ is present (= `greedy connection algorithm'). Note that the greedy connection rule has a simple interpretation: two aggregate vertices are connected whenever at least one microscopic connection exists between their constituents. Thus, the coarse-grained edge records the existence of a connection between two aggregate vertices, but not the identity of the microscopic connections.


\paragraph{Renormalisation transformation.}

We view $G^{(1)}$ as a \emph{level-$1$ renormalisation} of $G$, which is again an inhomogeneous random graph, on $n_1=n/m$ level-$1$ vertices. Clearly, in $G^{(1)}$ the pair of level-1 vertices $\{I^{(1)},J^{(1)}\}$ with $I^{(1)} \neq J^{(1)}$ is connected with probability 
\[
f^{\neq}_{I^{(1)},J^{(1)}}[X] = 1- \prod_{i \in I^{(1)},\,j \in J^{(1)}} \big(1-f^{\neq}(X_i,X_j)\big),
\]
while the level-1 vertex $I^{(1)}$ is connected to itself by a single self-loop with probability
\[
f^{=}_{I^{(1)},I^{(1)}}[X] = 1 - \prod_{ {i,j \in I^{(1)}} \atop {i<j} } \big(1-f^{\neq}(X_i,X_j)\big)
\prod_{ i \in I^{(1)}} \big(1-f^{=}(X_i)\big),
\] 
where the restriction $i<j$ is inserted in order to avoid double counting of the internal pairs. (Ordering of the vertices in $I^{(1)}$ can be done arbitrarily because of exchangeability in the product.) Clearly, in $G^{(1)}$ different pairs of level-$1$ vertices are connected independently. Note that $G^{(1)}$ is a \emph{deterministic} function of $G$, obtained via the greedy connection algorithm. All the randomness (weights and connections) sits in $G$ \emph{alone}.

It is convenient to write
\[
f^{\neq}(x,y) = 1-\ee^{-g^{\neq}(x,y)}, \qquad f^{=}(x) = 1-\ee^{-\tfrac12[g^{\neq}(x,x)+g^{=}(x)]},
\] 
with $g^{\neq}\colon\,\R_+ \times \R_+ \to \R_+$ and $g^{=}\colon\,\R_+ \to \R_+$. With this change of variables the above relations can be \emph{rewritten in the compact form}  
\[
g^{\neq}_{I^{(1)},J^{(1)}}[X] = \sum_{i \in I^{(1)},\,j \in J^{(1)}} g^{\neq}(X_i,X_j), \quad I^{(1)} \neq J^{(1)},
\qquad g^{=}_{I^{(1)}}[X] = \sum_{i \in I^{(1)}} g^{=}(X_i).
\]
The latter shows that $g^{\neq}$ and $g^=$ renormalise \emph{autonomously}, which is an important simplification. 

Note that the renormalisation transformation is such that, even when we start from connection functions that depend on the weights of the vertices only, i.e., $g^{\neq}_{ij}[X] = g^{\neq}(X_i,X_j)$ and $g^=_i[X] = g^=(X_i)$ for level-$0$ vertices $i,j$, this property \emph{does not carry over} to the aggregated connection functions. Indeed, $g^{\neq}_{I^{(1)},J^{(1)}}[X]$ and $g^{=}_{I^{(1)}}[X]$ in general depend on the two sets $X_{I^{(1)}} \equiv \{X_i\}_{i\in I^{(1)}}$ and $X_{J^{(1)}} \equiv \{X_j\}_{j\in J^{(1)}}$ of weights of the constituent vertices of the aggregated vertices $I^{(1)}$ and $J^{(1)}$. Thus, the transformation \emph{does not preserve} the class of connection functions that depend on the scalar weights only, a situation that is \emph{typical for renormalisation schemes}. One of our goals is to find out whether or not particular connection functions exist that remain inside the class, and whether these act as attractors for the iterated transformation. We will see that the answer is yes.

The following was proved in \cite{GLG} in a simpler case ($\eta=0$) and, for completeness, will be reproved in Section~\ref{ss.invlem} in full generality.

\begin{lemma}{\bf [Invariance property for the connection probability]}
\label{leminv}
For every $\delta \in (0,\infty)$ and $\eta \in [0,\infty)$, the choice 
\[
g^{\neq}_\delta(x,y) = \delta xy, \qquad g^{=}_\eta(x) = \eta x,
 \]
is invariant under addition of the weights in the aggregation. In other words, the renormalised graph has the same connection functions as the original graph, provided the weight of an aggregated vertex equals the sum of the weights of the constituent vertices, i.e., the aggregated vertex $I^{(1)}$ has weight $\bar{X}_{I^{(1)}} = \sum_{i \in I^{(1)}} X_i$. 
\end{lemma}

\noindent
In Section~\ref{ss.invlem} we will also check the following invariance property.

\begin{lemma}{\bf [Invariance property for positive stable weights]}
\label{lemstable}
Let $\gamma\in(0,1)$, and let $Y_1,\ldots,Y_m$ be i.i.d.\ standard positive $\gamma$-stable random variables, with Laplace transform
\[
\EE[\ee^{-\lambda Y_1}] = \ee^{-\lambda^\gamma}, \qquad \lambda\geq 0.
\]
Then
\[ 
m^{-1/\gamma}\sum_{i=1}^mY_i \stackrel{d}{=}Y_1.
\]
Consequently, for every $\delta\in(0,\infty)$ and $\eta\in[0,\infty)$, the inhomogeneous random graph with connection functions
\[
g^{\neq}_\delta(x,y)=\delta xy, \qquad g^{=}_\eta(x)=\eta x,
\]
and i.i.d.\ standard positive $\gamma$-stable weights is invariant in distribution under the renormalisation transformation, provided the aggregated weights are rescaled by $m^{-1/\gamma}$.
\end{lemma}

The restriction $\gamma\in(0,1)$ is natural because we work with positive one-sided stable weights. Stable laws with index $\gamma \in [1,2)$ do not give non-degenerate positive one-sided weights of this form.

\begin{lemma}{\bf [Invariance property for deterministic weights]}
\label{lemER}
For every $\delta\in(0,\infty)$, $\eta\in[0,\infty)$ and $w\in(0,\infty)$, the deterministic weight law concentrated at $w$ is invariant under averaging of weights: if $W_1,\ldots,W_m$ are deterministic random variables equal to $w$, then 
\[
m^{-1} \sum_{i=1}^m W_i \stackrel{d}{=}W_1,
\]
where in this case equivalence in distribution coincides with equality in value. Combined with the invariant connection functions $g^{\neq}_\delta(x,y) = \delta xy$ and $g^{=}_\eta(x,x) = \eta x$ identified in Lemma~\ref{leminv}, this gives a two-parameter homogeneous Erd\H{o}s--R\'enyi random graph with edge and self-loop  connection probability equal to
\[
p^{\neq} = 1-\ee^{-\delta w^2}, \qquad p^{=} = 1-\ee^{-\frac12(\delta w^2+\eta w)}.
\]
The above random graph is invariant in distribution under the renormalisation transformation, provided the aggregated weights are rescaled by $m^{-1}$.
\end{lemma}

\noindent
Thus, we have two different \emph{two-parameter families} of connection functions and weights that are invariant under the renormalisation map when weights are \emph{added} and \emph{rescaled}. We will show that each acts as an \emph{attractor} for iterations of the renormalisation map under appropriate scaling of the weights. As illustrated by the different expressions for the transformations of the random variables $Y_i$ and $W_i$ considered in Lemma~\ref{lemstable} and~\ref{lemER}, respectively, the required scaling of the weights is different for the two attractors.


\paragraph{Iteration of the renormalisation transformation.}

The renormalisation can be iterated, leading to a sequence of random graphs $\{G^{(\ell)}[X]\}_{0 \leq \ell \leq L}$ in which vertices get more and more aggregated, at each step into groups of equal size $m$. The number of vertices at level $\ell$ is therefore $n_\ell = n/m^{\ell}$. Note that $G^{(\ell)}$ is a \emph{deterministic} function of $G$, obtained by applying the greedy connection algorithm $\ell$ times. Again, all the randomness (weights and connections) sits in $G$ \emph{alone}.

Note that the iteration preserves the dependence of the aggregated connection functions on the full set of microscopic weights inside each aggregate vertex. Indeed, in $G^{(\ell)}$ the pair of level-$\ell$ vertices $\{I^{(\ell)},J^{(\ell)}\}$ with $I^{(\ell)} \neq J^{(\ell)}$ is connected with probability 
\[
f^{\neq}_{I^{(\ell)},J^{(\ell)}}[X] = 1- \prod_{i \in I^{(\ell)},\,j \in J^{(\ell)}} \big(1-f^{\neq}(X_i,X_j)\big),
\]
while the level-$\ell$ vertex $I^{(\ell)}$ is connected to itself by a single self-loop with probability
\[
f^{=}_{I^{(\ell)},I^{(\ell)}}[X] = 1 - \prod_{ {i,j \in I^{(\ell)}} \atop {i<j} } \big(1-f^{\neq}(X_i,X_j)\big)
\prod_{ i \in I^{(\ell)}} \big(1-f^{=}(X_i)\big).
\] 
This is compactly written as
\[
g^{\neq}_{I^{(\ell)},J^{(\ell)}}[X] = \sum_{i \in I^{(\ell)},\,j \in J^{(\ell)}} g(X_i,X_j),
\qquad g^{=}_{I^{(\ell)}}[X] = \sum_{i \in I^{(\ell)}} g^{=}(X_i).
\]

Since $n=m^L$, a total of $L$ iterations is possible, after which all vertices have been aggregated into a \emph{single vertex} at level $L$, i.e., $n_L=1$. At every level $\ell$ with $0 < \ell \leq L$ the connection probabilities a priori depend on \emph{all} the weights whose labels fall in the aggregated vertex sets, not just on the \emph{sum} of these weights. Only for the special case where $g = g_\delta$ does this reduction occur. Nonetheless, we may expect that after many iterations some sort of \emph{homogenisation} takes place, such that this reduction occurs at least \emph{approximately}, in the spirit of abstract renormalisation theory. The goal of the current paper is to identify choices of the connection functions for which such a simplification occurs after appropriate scaling. We will see that light-tailed regimes and heavy-tailed regimes lead to different scalings and different scaling limits.

Although we might loosen the restriction of the renormalisation being done in groups of vertices of equal size, this setup is easy and maintains exchangeability. When passing to the limit $\ell\to\infty$, we need that also $L\to\infty$. In order to achieve this, we may fix an infinite sequence of weights $\{X_i\}_{i \in \N}$ beforehand and embed $\{X_i\}_{i \in [n]}$ into this sequence for every $n$. We define $N_{L-\ell} = n_\ell = m^{L-\ell}$ and think of the setting where $N_{L-\ell} \gg 1$, so that even after $\ell$ iterations the renormalised graph still consists of a large number of level-$\ell$ vertices. In what follows, $n$ does not play much of a role. We only need it to be large enough. Our focus will be on what the renormalised random graph looks like for \emph{large but finite} $\ell$, even though our main theorems are phrased with $\ell \to \infty$ to get sharp and universal limiting results. Throughout the paper $m$ is kept fixed.


\subsection{Main theorems}
\label{ss.theorems}

We apply the renormalisation transformation $\ell$ times and investigate what happens when $\ell \to \infty$. In order to end up with a non-degenerate graph limit, we need to start from a graph in which the edges and the loops are \emph{diluted} in an $\ell$-dependent way. Here and below, by dilution we mean that the microscopic (= level-0) connection intensities are allowed to depend on the number $\ell$ of renormalisation steps and tend to zero as $\ell\to\infty$. This should not be confused with sparsity in the sense of bounded average degree. Rather, for each target scale $\ell$ we start from a microscopic graph whose connection function is weakened by an $\ell$-dependent factor. More explicitly, if 
\[
f^{\neq}(x,y) = 1-\ee^{-g^{\neq}(x,y)}, \qquad 
f^{=}(x) = 1-\ee^{-\tfrac12 [g^{\neq}(x,x)+g^{=}(x)]},
\]
then the $\ell$-diluted version has rescaled connection functions 
\[
\mathfrak{f}_\ell^{\neq}(x,y) 
= 1-\ee^{-\mathfrak{g}_\ell^{\neq}(x,y)},
\qquad \mathfrak{f}_\ell^{=}(x) 
= 1-\ee^{-\tfrac12[\mathfrak{g}_\ell^{\neq}(x,x)+\mathfrak{g}_\ell^{=}(x)]},
\]
where
\[
\mathfrak{g}^{\neq}_\ell(x,y) = a^{\neq}_\ell g^{\neq}(x,y), \qquad 
\mathfrak{g}^{=}_\ell(x) = a^{=}_\ell g^{=}(x),
\]
for some appropriately chosen $a^{\neq}_\ell \downarrow 0$ and $a^{=}_\ell \downarrow 0$ as $\ell \to \infty$. For instance, for large $\ell$ and bounded $x,y$, we have
\[
\mathfrak{f}^{\neq}_\ell(x,y) \approx a^{\neq}_\ell g^{\neq}(x,y),
\]
and each microscopic (= level-$0$) edge is individually unlikely. The reason for this scaling is that a level-$\ell$ edge represents the union of $m^{2\ell}$ possible microscopic edges between two blocks of size $m^\ell$. Without weakening the microscopic connection probabilities, the greedy renormalisation would make the level-$\ell$ graph complete in the limit as $\ell\to\infty$. The dilution is therefore the \emph{scaling window} in which the renormalised graph has a non-trivial limit.


\paragraph{Light-tailed regime.}
Write $P(G^{(\ell)}[X])$ to denote the law of the \emph{inhomogeneous random graph} $G^{(\ell)}[X]$ that is obtained after carrying out $\ell$ renormalisations of the graph with \emph{rescaled connection functions} $\mathfrak{g}^{\neq}_\ell$ and $\mathfrak{g}^{=}_\ell$ given by
\[
\mathfrak{g}^{\neq}_\ell = (m^{2\ell})^{-1}\, g^{\neq}, \qquad  \mathfrak{g}^{=}_\ell = (m^{\ell})^{-1}\, g^{=},
\]
\emph{conditional} on the weights $X$ (= `quenched setting'). This law lives on the set of all graphs with $N_{L-\ell}$ vertices and no multiple edges or multiple loops. Write $P(G_{\ell,p})$ to denote the law of 
\begin{itemize}
\item
the \emph{homogeneous Erd\H{o}s--R\'enyi random graph} $G_{\ell,p}$ with $N_{L-\ell}$ vertices, edge probability $p^{\neq} \in (0,1)$, loop probability $p^{=} \in (0,1)$, and $p=(p^{\neq},p^{=})$. 
\end{itemize}

Let
\[
\Delta\Big(P(G^{(\ell)}[X]),P(G_{\ell,p})\Big)
= S^{\neq}\Big(P(G^{(\ell)}[X]) ~\Big|~ P(G_{\ell,p})\Big) + S^{=}\Big(P(G^{(\ell)}[X]) ~\Big|~ P(G_{\ell,p})\Big),
\]
where $S^{\neq}$ and $S^{=}$ are the specific relative entropies for edges and self-loops defined in Section~\ref{ss.relentr}.

\begin{theorem}{\bf [Averaged homogeneous attractor]}
\label{thm1}
Suppose that
\[
(\Box) \qquad
\begin{aligned}
&0 < \EE_X\big[g^{\neq}(X_1,X'_1)\big] < \infty,\\
&0 \leq \EE_X\big[g^{=}(X_1)\big] < \infty,\\
&k^{-2} \sum_{i=1}^k g^{\neq}(X_i,X_i)  \xrightarrow[k\to\infty]{\mathbb P} 0,
\end{aligned}
\]
where \(X'_1\) is an independent copy of \(X_1\) and $(X_i)_{i\in\N}$ are i.i.d.\ copies of $X_1$. Put
\[
q^{\neq} = \EE_X\big[g^{\neq}(X_1,X'_1)\big],
\qquad
q^{=}=\EE_X\big[g^{=}(X_1)\big],
\]
and
\[
p^{\neq} = 1-\ee^{-q^{\neq}},
\qquad
p^{=} = 1-\ee^{-\frac12(q^{\neq}+q^{=})}.
\]
Then
\[
\lim_{\ell \to \infty} \sup_{L \geq \ell}\,\EE_X\left[\Delta\Big(P(G^{(\ell)}[X]),P(G_{\ell,p})\Big)\right] = 0
\]
with \(p=(p^{\neq},p^{=})\).
\end{theorem}

\begin{corollary}{\bf [Product kernel with finite first effective moment]}
\label{cor:product-light}
Suppose that, for some $B,\beta\in(0,\infty)$ and $C,K^{\neq},K^{=}<\infty\in[0,\infty)$,
\[
\big|g^{\neq}(x,y)-B(xy)^\beta\big|\leq K^{\neq}, \quad x,y\in\R_+,
\qquad \big|g^{=}(x)-Cx^\beta\big|\leq K^{=}, \quad x\in\R_+.
\]
Suppose further that
\[
0<\EE_X[X_1^\beta]<\infty.
\]
Then Theorem~\ref{thm1} applies. In particular, if
\[
\PP(X_1>x)\sim Ax^{-\alpha}, \qquad x\to\infty,
\]
then the product connection function is attracted to the homogeneous Erd\H{o}s--R\'enyi random graph whenever $\alpha>\beta$.
\end{corollary}


\paragraph{Heavy-tailed regime.}

Fix $\gamma \in (0,1)$. Write $P(G^{(\ell)})$ to denote the law of the random graph $G^{(\ell)}$ that is obtained after carrying out $\ell$ renormalisations of the graph with \emph{rescaled connection functions} $\mathfrak{g}^{\neq}_\ell$ and $\mathfrak{g}^{=}_\ell$ given by
\[
\mathfrak{g}^{\neq}_\ell = (m^{2\ell})^{-1/\gamma} g^{\neq},
\qquad \mathfrak{g}^{=}_\ell = (m^{\ell})^{-1/\gamma}\,g^{=},
\]  
\emph{unconditional} on the weights $X$ (= `annealed setting'). Write $P(G_{\ell,\gamma,\delta,\eta})$ to denote the law of 
\begin{itemize}
\item
the \emph{inhomogeneous stable law random graph} $G_{\ell,\gamma,\delta,\eta}$ with $N_{L-\ell}$ vertices, with weights given by i.i.d.\ standard $\gamma$-stable law random variables $\{Y_I\}_{I \in [N_{L-\ell}]}$, and with connection functions $g^{\neq}_\delta$ and $g^{=}_\eta$.
\end{itemize}
The standard $\gamma$-stable random variable $Y_1$ on $\R_+$ has Laplace transform $\EE\big[\ee^{-\lambda Y_1}\big] = \ee^{-\lambda^\gamma}$, $\lambda \in \R_+$. It has tail probabilities $\PP(Y_1 > x) \sim x^{-\gamma}/\Gamma(1-\gamma)$, $x \to \infty$, and $\PP(Y_1 \leq x) = o(1)\,\ee^{-x^{-\gamma}}$, $x \downarrow 0$. See \cite[Section XIII.6]{F} for more details. 

\begin{theorem}{\bf [Averaged stable-law attractor]}
\label{thm2}
Suppose that
\[
(\blacksquare) \qquad 
\begin{aligned}
&\PP(X_1 > x) \sim A x^{-\alpha}, \quad x \to \infty,\\
&\big|g^{\neq}(x,y) - B(xy)^\beta\big| \leq K^{\neq}, \quad x,y \in \R_+,\\
&\big|g^{=}(x) - Cx^\beta\big| \leq K^{=}, \quad x \in \R_+,
\end{aligned}
\]
with $\alpha \in (0,\infty)$ a tail exponent, $\beta \in (0,\infty)$ a growth exponent, and $A,B,C,K^{\neq},K^{=} \in (0,\infty)$ constants. Suppose, in addition, that the law of $X_1$ is absolutely continuous. If $\alpha<\beta$, then there exists a coupling of $(X,Y)$ via a sequence $U$ of i.i.d.\ $\mathrm{UNIF}[0,1]$ random variables, written $(X(U),Y(U))$, such that  
\[
\lim_{\ell \to \infty} \sup_{L \geq \ell}\,\EE_U\left[\Delta\Big(P(G^{(\ell)}[X(U)]),P(G_{\ell,\gamma,\delta,\eta}[Y(U)])\Big)\right] = 0
\]
with $\gamma = \frac{\alpha}{\beta}$, $\delta = B\big(A\Gamma(1-\gamma)\big)^{2/\gamma}$ and $\eta = C\big(A\Gamma(1-\gamma)\big)^{1/\gamma}$.
\end{theorem}


\subsection{Discussion and open problems}
\label{ss.discussion}

{\bf 1.} 
To the best of our knowledge, the present paper is the first to provide a fully rigorous renormalisation analysis of a large class of inhomogeneous random graphs with random weights, both in a light-tailed regime and in a heavy-tailed regime. Universality classes and domains of attraction are identified in both regimes, exhibiting different scaling behaviour. We focus on a particular algorithm for the renormalisation, namely, the `greedy connection algorithm', although the tools we use in our analysis are flexible (see items 9--10 below). From a physical perspective, identifying previously known \emph{fixed points} of a renormalisation transformation as \emph{attractors} of the associated flow is a powerful step forward, because it provides an explicit and natural mechanism (namely, the iterated change of resolution level) according to which effective aggregate models of networks arise, irrespectively of microscopic details that are irrelevant at the macroscopic scale.

\medskip\noindent
{\bf 2.}  Theorem~\ref{thm1} establishes universality for the renormalisation in the light-tailed regime: after renormalisation and appropriate scaling of the connection functions, $G^{(\ell)}[X]$ and $G_{\ell,p}$ are close to each other in the sense of specific relative entropy when $\ell\to\infty$, uniformly in $L\geq \ell$, provided
\[
q^{\neq}=\EE_X[g^{\neq}(X_1,X'_1)],
\qquad
q^{=}=\EE_X[g^{=}(X_1)],
\]
and
\[
p^{\neq}=1-\ee^{-q^{\neq}},
\qquad
p^{=}=1-\ee^{-\frac12(q^{\neq}+q^{=})}.
\]
The term $q^{\neq}$ comes from level-$0$ edges inside a block, while the term $q^{=}$ comes from level-$0$ self-loops. The scaling limit is \emph{homogeneous}, i.e., there are \emph{no} weights involved, and the claim holds in expectation. Both Theorem~\ref{thm1} and Theorem~\ref{thm2} are averaged relative-entropy statements. In Theorem~\ref{thm1}, the law of the target graph is deterministic, because the limiting Erd\H{o}s--R\'enyi graph has deterministic edge and loop probabilities. Convergence to zero of the specific relative entropy means that around a \emph{typical edge} and a \emph{typical vertex} the probability laws are asymptotically the same. The domain of attraction given by $(\Box)$ is large: it includes arbitrary $g$ combined with bounded $X_1$, as well as $g$ bounded by a bi-variate polynomial and $X_1$ having a finite moment of an appropriate order, as illustrated in Corollary~\ref{cor:product-light}.  

\medskip\noindent
{\bf 3.}
Theorem \ref{thm2} establishes the same in the heavy-tailed regime, but with a scaling that depends on the tail exponent of the distribution of $X_1$ and the growth exponent of $g^{\neq}$, $g^{=}$, and with a scaling limit that is \emph{inhomogeneous}. The limit is invariant in the sense of Lemma~\ref{leminv} and has stable law random weights with exponent $\gamma$. In Theorem~\ref{thm2}, the target graph law is itself random through the stable weights. The coupling in Theorem~\ref{thm2} is used to realise the original and limiting weight families on a common probability space. For product kernels $g^{\neq} (x,y) = B(xy)^\beta + O(1)$ and $g^{=}(x) = Cx^\beta + O(1)$, the case $\alpha>\beta$ is covered by Theorem~\ref{thm1}, as noted in Corollary~\ref{cor:product-light}, while the case $\alpha<\beta$ is covered by Theorem~\ref{thm2}. Thus, the borderline between the homogeneous and the stable-law attractor is governed by the finiteness or infiniteness of the first effective moment $\mathbb E[X_1^\beta]$. The domain of attraction given by $(\blacksquare)$ is fair, but far from exhaustive. It remains open to obtain a full classification in the heavy-tailed regime. For instance, what happens when $g^{\neq}(x,y) = B(x+y)^{\beta_1} + O(1)$, $x,y \to \infty$, and $g^{=}(x) = Cx^{\beta_2} + O(1)$, $x \to \infty$?

\medskip\noindent
{\bf 4.} 
Previous research has shown that several random graph models defined by the same (random) vertex weights but different connection functions are either asymptotically \emph{equivalent} \cite{BJR} or mutually \emph{contiguous} \cite{J} under conditions that ultimately require the finiteness of certain moments of the weights. For instance, the three models defined by the connection probabilities
\[
f^{\neq}_1(x,y)=\kappa(x,y)\wedge 1, \qquad
f^{\neq}_2(x,y)=\frac{\kappa(x,y)}{1+\kappa(x,y)}, \qquad
f^{\neq}_3(x,y)=1-\ee^{-\kappa(x,y)},
\]
for the same kernel $\kappa(x,y)$, are asymptotically equivalent or mutually contiguous when certain conditions are met \cite{BJR,J}. In our setting, the corresponding connection functions are
\[
g^{\neq}_1(x,y)=-\log[1-\kappa(x,y)\wedge 1], \qquad
g^{\neq}_2(x,y)=\log[1+\kappa(x,y)], \qquad
g^{\neq}_3(x,y)=\kappa(x,y),
\]
with the convention $-\log 0=\infty$. Thus the relevant object for the renormalisation is not the connection probability $f$ itself, but the connection function $g^{\neq}=-\log(1-f^{\neq})$.
\begin{itemize}
\item
In the light-tailed regime, if the corresponding intensities satisfy $(\Box)$, then Theorem~\ref{thm1} applies to each of these models and gives a homogeneous Erd\H{o}s--R\'enyi attractor, with parameter $p^{\neq}$ determined by the corresponding value of
\[
\EE_X[g^{\neq}_i(X_1,X'_1)].
\]
For the model defined by $f^{\neq}_1$, this requires in particular that $\kappa(x,y)<1$ on the relevant support, or at least that $-\log[1-\kappa(X_1,X'_1)\wedge 1]$ be integrable. Otherwise $g^{\neq}_1$ may be infinite with positive probability and Theorem~\ref{thm1} does not apply.
\item
In the heavy-tailed regime, the moment assumptions that guarantee asymptotic equivalence or mutual contiguity may fail, and models that are comparable at level-$0$ may lie in different domains of attraction under the renormalisation map. For instance, take
\[
\kappa(x,y)=B(xy)^\beta, \qquad B,\beta\in(0,\infty),
\]
and assume that $\PP(X_1>x)\sim Ax^{-\alpha}$ as $x\to\infty$. For the model $f^{\neq}_2$, the associated connection function is $g^{\neq}_2(x,y) = \log\bigl(1+B(xy)^\beta\bigr)$, which satisfies the integrability condition in $(\Box)$. Indeed,
\[
\log\bigl(1+B(X_1X'_1)^\beta\bigr) \leq C+\beta\log(1+X_1)+\beta\log(1+X'_1)
\]
for a constant $C<\infty$, and $\EE_X[\log(1+X_1)]<\infty$ for regularly varying $X_1$ with tail exponent $\alpha>0$. 
Similarly,
\[
\EE_X\bigl[\log(1+BX_1^{2\beta})\bigr]<\infty.
\]
Hence $f^{\neq}_2$ falls within the light-tailed theorem (when $\EE_X[g^{=}_2(X)]<\infty$) and flows to a homogeneous Erd\H{o}s--R\'enyi attractor. For the model $f^{\neq}_3$, however, the associated connection function is
\[
g^{\neq}_3(x,y) = B(xy)^\beta.
\]
When $\alpha<\beta$, the random variable $X_1^\beta$ has tail exponent $\gamma=\alpha/\beta\in(0,1)$, and the hypotheses $(\blacksquare)$ of Theorem~\ref{thm2} are satisfied (when the condition for $g^{=}_3$ is in force). Therefore $f^{\neq}_3$ is attracted to the inhomogeneous $\gamma$-stable attractor. This shows that $f^{\neq}_2$ and $f^{\neq}_3$, although built from the same kernel $\kappa$, need not have the same attractor in the heavy-tailed regime. 
\item 
In the particular case where $\kappa(x,y) = \delta xy$, the three models in item 4 have special meanings: 
\begin{itemize}
\item
$f^{\neq}_1(x,y) = \delta xy\wedge 1$ is the Chung--Lu model \cite{CL}, 
\item
$f^{\neq}_2(x,y) = \frac{\delta xy}{1+\delta xy}$ is the canonical (or maximum-entropy) configuration model \cite{CCM},
\item
$f^{\neq}_3(x,y) = 1-\ee^{-\delta xy}$ is the invariant model in Lemma~\ref{leminv}. 
\end{itemize}
Thus, in the light-tailed regime, these models may belong to the same homogeneous universality class, whereas in the heavy-tailed regime they may be separated into different domains of attraction. 
\end{itemize}

\medskip\noindent
{\bf 5.}
Another notion of equivalence between different models has been established in the context of Exponential Random Graphs (ERGs)~\cite{Chatterjee}, which are defined as random graph ensembles with a probability distribution that maximises Shannon entropy~\cite{CCM} while enforcing a desired set of constraints in terms of expected network properties. In particular, it has been shown~\cite{Bhamidi,Diaconis} that, in the dense regime, various ERGs defined through different constraints are asymptotically indistinguishable from the Erd\H{o}s--R\'enyi model (which is itself an ERG, defined through a given expected link density). Those results indicate that it is hard to perturb the homogeneous Erd\H{o}s--R\'enyi model in the dense regime via the addition of other constraints. Under the different lens of renormalisation, our results show that: on one hand, the dense Erd\H{o}s--R\'enyi model is indeed a homogeneous attractor for a large class of microscopic models (Theorem \ref{thm1}); on the other hand, a different dense graph, clearly distinguishable from the Erd\H{o}s--R\'enyi one, exists and acts as an inhomogeneous attractor for a separate class of models (Theorem \ref{thm2}).

\medskip\noindent
{\bf 6.} 
The scaling of the connection functions has a simple probabilistic interpretation. In the light-tailed regime, no single microscopic (= level-$0$) vertex dominates a block. Many small contributions average out, and the coarse-grained graph is effectively homogeneous. This corresponds to a law of large numbers universality class: microscopic heterogeneity is washed out under aggregation. In the heavy-tailed regime, however, a few large weights may dominate a block. Aggregation, even under the greedy rule, does not erase this heterogeneity. Instead, the effective block weights converge to a stable law, and the coarse-grained graph remains inhomogeneous. The word \emph{dilute} should be understood in this sense. It refers to weakening of individual microscopic connection probabilities, and not to sparsity in the sense of bounded average degree. In the light-tailed regime, a typical microscopic connection probability is of order $m^{-2\ell}$. Since the original graph has $n=m^L$ vertices, the microscopic expected degree is typically of order $n m^{-2\ell}=m^{L-2\ell}$, up to constants depending on the connection function and the weight distribution. Hence the microscopic graph has bounded average degree only in the regime $L-2\ell = O(1)$. In general, the microscopic graph is diluted because individual edge probabilities are small, but it need not be sparse. A concrete example is obtained by taking, in the light-tailed regime, the rescaled invariant edge probability
\[
\mathfrak{f}^{\neq}_{\ell}(x,y)
= 1 - \exp\left(-\frac{\delta xy}{m^{2\ell}}\right).
\]
Then each pair of vertices has only a small probability of being connected. However, between two level-$\ell$ blocks $I^{(\ell)},J^{(\ell)}$ the effective connection intensity is
\[
\mathfrak{g}^{\neq}_\ell (X_{I^{(\ell)}},X_{J^{(\ell)}}) 
= \frac{\delta}{m^{2\ell}} \sum_{i\in I^{(\ell)},j\in J^{(\ell)}}X_iX_j
= \delta \left(\frac{1}{m^\ell}\sum_{i\in I^{(\ell)}}X_i\right) \left(\frac{1}{m^\ell}\sum_{j\in J^{(\ell)}}X_j\right).
\]
If $X_1$ has finite mean, then this product converges to $\delta(\mathbb E[X_1])^2$, and the coarse-grained graph becomes a homogeneous Erd\H{o}s--R\'enyi random graph with edge probability
\[
1-\exp(-\delta(\mathbb E[X_1])^2).
\]
This example shows that the dilution is not merely a technical device: it is the scale at which many weak microscopic interactions combine into a macroscopic interaction of order one.

\medskip\noindent
{\bf 7.}
The scaling of the connection functions $g^{\neq},g^{=}$ to the connection functions $\mathfrak{g}^{\neq}_\ell,\mathfrak{g}^{=}_\ell$ is important in order to end up with a \emph{non-trivial graph limit}. Indeed, under the aggregation, weights `accumulate' and the graph becomes `denser'. By starting with a dilute graph where the edges have a density of order $m^{-2\ell}$, respectively, $(m^{-2\ell})^{1/\gamma}$, and the loops have a density of order $m^{-\ell}$, respectively, $(m^{-\ell})^{1/\gamma}$, the initial graph is \emph{pre-engineered} so that after $\ell$ renormalisations it acquires an edge density and a loop density of order $1$. Consequently, the whole renormalisation flow acts in the space of \emph{dense graphs}, from dilute to non-dilute. Without the rescaling, the graph limit would be the complete graph, which is not interesting.

\medskip\noindent
{\bf 8.}
Relative entropy satisfies the contraction property: $S(TP \mid TQ) \leq S(P \mid Q)$ for any two probability measures $P,Q$ to which an aggregation map $T$ is applied, i.e., a surjection \cite{CT}. It therefore follows that $\ell \mapsto S(P(G^{(\ell)}[X]) \mid P(\bar{G}^{(\ell)}[\bar{X}]))$ is non-increasing for any two starting graphs $G[X],\bar{G}[\bar{X}]$.  Indeed, for fixed weights $X$, the map sending $G^{(\ell-1)}[X]$ to $G^{(\ell)}[X]$ is a deterministic map on graph configurations. Hence the law $P(G^{(\ell)}[X])$ is the push-forward of $P(G^{(\ell-1)}[X])$ under this map. Therefore relative entropy is a natural quantity to analyse under renormalisation.

\medskip\noindent
{\bf 9.}
It is interesting to pursue other types of aggregation as well. For instance, we may connect two level-$1$ vertices if and only if \emph{every vertex in one group is connected to every vertex in the other group} (= `lazy connection algorithm'). In that case the connection probabilities renormalise as
\[
\begin{aligned}
f^{\neq}_{I^{(1)},J^{(1)}}[X] &= \prod_{i \in I^{(1)},\,j \in J^{(1)}} f^{\neq}(X_i,X_j), \qquad I^{(1)} \neq J^{(1)},\\
f^{=}_{I^{(1)},I^{(1)}}[X] &= \prod_{ {i,j \in I^{(1)}} \atop {i<j} } f^{\neq}(X_i,X_j) \prod_{ i \in I^{(1)}} f^{=}(X_i),
\end{aligned} 
\]
The proper change of variables is
\[
f^{\neq}(x,y) = \ee^{-g^{\neq}(x,y)}, \qquad
f^{=}(x) = \ee^{-\tfrac{1}{2} [g^{\neq}(x,x)+g^{=}(x)]}.
\]
Iteration of the renormalisation gives
\[
g^{\neq}_{I^{(\ell)},J^{(\ell)}}[X] = \sum_{i \in I^{(\ell)},\,j \in J^{(\ell)}} g^{\neq}(X_i,X_j),
\qquad g^{=}_{I^{(\ell)},I^{(\ell)}}[X] = \sum_{i \in I^{(\ell)}} g^{=}(X_i).
\]
The invariant connection functions are
\[
g^{\neq}_\delta(x,y) = \delta xy,\qquad g^{=}_\eta(x)=\eta x.
\]
In the present paper we will not pursue such generalisations, even though the arguments developed in Sections~\ref{s.renor}--\ref{s.proofs} can in principle be carried over.

\medskip\noindent
{\bf 10.}
Another example, not based on an aggregation algorithm, is where we connect two level-$1$ vertices with a probability that equals the \emph{maximum of the connection probabilities of the constituent pairs} in the two groups. In that case the connection probabilities renormalise as
\[
\begin{aligned}
f^{\neq}_{I^{(1)},J^{(1)}}[X] &= \max_{i \in I^{(1)},\,j \in J^{(1)}} f^{\neq}(X_i,X_j), \qquad I^{(1)} \neq J^{(1)},\\
f^{=}_{I^{(1)}}[X] &= \max_{i \in I^{(1)}} f^{=}(X_i).
\end{aligned}
\]
Iteration of the renormalisation gives
\[
\begin{aligned}
f^{\neq}_{I^{(\ell)},J^{(\ell)}}[X] &= \max_{i \in I^{(\ell)},\,j \in J^{(\ell)}} f^{\neq}(X_i,X_j), \qquad I^{(\ell)} \neq J^{(\ell)},\\
f^{=}_{I^{(\ell)}}[X] &= \max_{i \in I^{(1)}} f^{=}(X_i).
\end{aligned}
\]
If $f^{\neq},f^{=}$ are non-decreasing in both arguments, then
\[
\begin{aligned}
\max_{i \in I^{(\ell)},\,j \in J^{(\ell)}} f^{\neq}(X_i,X_j) &= f^{\neq}(\tilde{X}_{I^{(\ell)}},\tilde{X}_{J^{(\ell)}}),
\qquad I^{(\ell)} \neq J^{(\ell)},\\
\max_{i \in I^{(\ell)}} f^{=}(X_i) &= f^{=}(\tilde{X}_{I^{(\ell)}}),
\end{aligned}
\]
with $\tilde{X}_{I^{(\ell)}} = \max_{i \in I^{(\ell)}} X_i$. In other words, the aggregation no longer sums the weights, but maximises the weights.
 
For heavy-tailed weights the difference between sums and maxima of constituent weights should be small, so that different models should exhibit similar scaling behaviour. More in general, it is interesting to compare different renormalisation maps applied to the same starting graph. 

\medskip\noindent
{\bf 11.}
What happens when the weights are dependent? A natural extension is to assume that the infinite sequence $(X_i)_{i\geq1}$ is exchangeable, in the sense that its joint law is invariant under finite permutations. By de Finetti's theorem, this amounts to weights that are conditionally i.i.d. Their dependence may affect the attractors of the renormalisation. Again we may ask what happens for multiscale stochastic block models, configuration-type models, or weights arising from more general infinitely divisible laws.


\subsection{Outline of the remainder}
\label{ss.outline}

Section~\ref{s.renor} analyses the renormalisation transformation by identifying its iterates and establishing bounds on the relative entropy of the laws of the renormalised connection functions. Section~\ref{s.proofs} provides the proofs of the invariance lemma and the two scaling theorems.


\section{Renormalisation transformation}
\label{s.renor}

In Section~\ref{ss.iter} we identify how the renormalisation transformation iterates. In Section~\ref{ss.relentr} we estimate specific relative entropies. In Section~\ref{ss.pertrelentr} we estimate perturbations of specific relative entropies. The results will be used in Section~\ref{s.proofs} to prove our main theorems. Appendix~\ref{App} collects a few basic facts about stable laws, domains of attraction and coupling.   


\subsection{Iteration}
\label{ss.iter}

The renormalisation transformation defined in Section~\ref{ss.map} can be iterated $L$ times when $n=m^L$. After $L$ iterations the renormalised graph $G^{(L)}$ consists of a \emph{single vertex}.
For $0 < \ell \le L$, aggregate level-$(\ell-1)$ vertices into disjoint groups of size $m$ each, referred to as level-$(\ell)$ vertices, to obtain an aggregated graph $G^{(\ell)}$ with connection function $f^{(\ell)}$ obtained iteratively as
\[
\begin{aligned}
f^{\neq}_{I^{(\ell)},J^{(\ell)}}[X] 
&= 1- \prod_{I^{(\ell-1)} \in I^{(\ell)},\,J^{(\ell-1)} \in J^{(\ell)}} \big(1-f^{\neq}_{I^{(\ell-1)},J^{(\ell-1)}}[X]\big), \quad I^{(\ell)} \neq J^{(\ell)},\\
f^{=}_{I^{(\ell)},I^{(\ell)}}[X] 
&= 1 - \prod_{ {I^{(\ell-1)}, J^{(\ell-1)} \in I^{(\ell)}} \atop {I^{(\ell-1)} < J^{(\ell-1)}} } \big(1-f^{\neq}_{I^{(\ell-1)},J^{(\ell-1)}}[X]\big)
\prod_{I^{(\ell-1)} \in I^{(\ell)}} \big(1-f^{=}_{I^{(\ell-1)}}[X]\big),
\end{aligned}
\] 
which after the change of variables reads
\[
\begin{aligned}
g^{\neq}_{I^{(\ell)},J^{(\ell)}}[X] 
&= \sum_{I^{(\ell-1)} \in I^{(\ell)},\,J^{(\ell-1)} \in J^{(\ell)}} 
g^{\neq}_{I^{(\ell-1)},J^{(\ell-1)}}[X],\\
g^{=}_{I^{(\ell)}}[X] 
&= \sum_{I^{(\ell-1)} \in I^{(\ell)}}
g^{=}_{I^{(\ell-1)}}[X].
\end{aligned} 
\]
When the right-hand sides of the above expressions are recursively expanded down to the microscopic level (= level-$0$) where $I^{(0)}=i$ and $J^{(0)}=j$, we obtain the iterated transformation equations shown in Section~\ref{ss.map}. 

As already noted, while the connection functions $g^{\neq}_{i,j}[X]=g^{\neq}(X_i,X_j)$ at level $\ell=0$ are assumed to depend on $X$ only through the two vertex weights $X_i$ and $X_j$, this is in general not true for levels $\ell>0$, where we can instead write 
\[
g^{\neq}_{I^{(\ell)},J^{(\ell)}}[X]=g^{\neq}(X_{I^{(\ell)}},X_{J^{(\ell)}}),
\quad X_{I^{(\ell)}} \equiv \{X_i\}_{i \in I^{(\ell)}}
\]
which depends on the fulls sets $X_{I^{(\ell)}}$, $X_{J^{(\ell)}}$ of microscopic weights of the constituent vertices of the aggregate vertices $I^{(\ell)}$ and $J^{(\ell)}$.

For the special case $g^{\neq}=g^{\neq}_\delta$, $g^{=}=g^{=}_\eta$, iteration gives
\[
g_\delta^{\neq}(X_{I^{(\ell)}},X_{J^{(\ell)}}) = \delta \bar{X}_{I^{(\ell)}} \bar{X}_{J^{(\ell)}},
\qquad g_\eta^{=}(X_{I^{(\ell)}}) = \eta \bar{X}_{I^{(\ell)}},
\]
with $\bar{X}_{I^{(\ell)}} = \sum_{i \in I^{(\ell)}} X_i$. For this invariant case, the connection functions $g^{\neq}_{I^{(\ell)},J^{(\ell)}}[X]$, $g^{=}_{I^{(\ell)}}[X]$ therefore still depend only on the two (renormalised) weights $\bar{X}_{I^{(\ell)}}$, $\bar{X}_{J^{(\ell)}}$ of the end-point aggregated vertices $I^{(\ell)}$ and $J^{(\ell)}$.


\subsection{Relative entropy}
\label{ss.relentr}

To measure how close a graph measure $P_f(G)$ is to another graph measure $P_{f'}(G)$ on the same set $\mathcal{G}_N$ of graphs with $N$ vertices, we consider the relative entropy
\[
H(P_f\mid P_{f'}) = \sum_{G\in\mathcal{G}_N}P_f(G)\log\frac{P_f(G)}{P_{f'}(G)} 
\]
For a random graph with $\binom{N}{2}$ independent edges and $N$ independent self-loops, the graph measure factorises into $\binom{N}{2} + N = \binom{N+1}{2}$ Bernoulli variables. Consequently, the relative entropy separates additively into such variables as 
\[
H(P_f \mid P_{f'}) = \sum_{1 \leq i \le j \leq N} 
\left[f_{ij} \log \left(\frac{f_{ij}}{f'_{ij}}\right) + (1-f_{ij}) \log \left(\frac{1-f_{ij}}{1-f'_{ij}}\right)\right]
\]
(where $f_{ij}$ and $f'_{ij}$ are the connection probabilities between $i$ and $j$ under $P_f$ and $P_{f'}$ respectively).
The above expression takes into account both edges between distinct pairs of vertices ($i<j$) and self-loops at individual vertices ($i=j$). To characterise asymptotic convergence to the attractor(s), we could consider the \emph{specific relative entropy} (= relative entropy density)
\[
S(P_f\mid P_{f'})  = \binom{N+1}{2}^{-1} \sum_{1 \leq i \le j \leq N} 
\left[f_{ij} \log \left(\frac{f_{ij}}{f'_{ij}}\right) + (1-f_{ij}) \log \left(\frac{1-f_{ij}}{1-f'_{ij}}\right)\right].
\]
Although this would be sufficient to characterise convergence, we note that the self-loop contribution, consisting of $N$ terms in the sum, would be overruled by the overall normalising factor. Therefore, to gain proper control on the convergence of the laws of the edges and of the self-loops separately, we first define two separately normalised relative entropies $S^{\neq}(P_f \mid P_{f'})$ and $S^{=}(P_f \mid P_{f'})$ for edges and self-loops respectively, i.e.,
\[
\begin{aligned}
S^{\neq}(P_f \mid P_{f'}) &= \binom{N}{2}^{-1} \sum_{1 \leq i < j \leq N} 
\left[f_{ij} \log \left(\frac{f_{ij}}{f'_{ij}}\right) + (1-f_{ij}) \log \left(\frac{1-f_{ij}}{1-f'_{ij}}\right)\right],\\ 
S^{=}(P_f \mid P_{f'}) &= N^{-1} \sum_{1 \leq i \leq N} \left[f_{ii} \log \left(\frac{f_{ii}}{f'_{ii}}\right) 
+ (1-f_{ii}) \log \left(\frac{1-f_{ii}}{1-f'_{ii}}\right)\right], 
\end{aligned} 
\]
and afterwards combine them into the single quantity
\[
\Delta(P_f,P_{f'}) = S^{\neq}(P_f \mid P_{f'}) + S^{=}(P_f \mid P_{f'}).
\]
Note that, while $\Delta(P_f,P_{f'})$ is not a relative entropy \emph{per se}, its vanishing is sufficient to ensure the separate convergence of the law of the edges and the law of the self-loops encoded in $P_f$ to the corresponding laws encoded in $P_{f'}$.

In terms of $g_{ij} = -\log(1-f_{ij})$ and $g'_{ij} = -\log(1-f'_{ij})$ we have the following estimate. In what follows we use the following bounds, valid for $U,V \in (0,1)$:
\[
(\ast) \qquad (1-U) \log \left(\frac{1-U}{1-V}\right) = (1-U) \log \left(1 + \frac{V-U}{1-V}\right)
\left\{\begin{array}{ll}
\leq (1-U)\,\frac{V-U}{1-V},\\[0.2cm]
\geq V-U.
\end{array}
\right.
\]

\begin{lemma}
\label{lem:entrbd1}
$0 \leq H^{\neq}(P_f \mid P_{f'}) \leq d^{\neq}(g,g')$ and $0 \leq H^{=}(P_f \mid P_{f'}) \leq d^{=}(g,g')$, where
\[
d^{\neq}(g,g') = \binom{N}{2}^{-1} \sum_{1 \leq i < j \leq N} \zeta(g_{ij},g'_{ij}), \qquad
d^{=}(g,g') = N^{-1} \sum_{1 \leq i \leq N} \zeta(g_{ii},g'_{ii}),
\]
with
\[
(\heartsuit_1) \qquad  \zeta(a,b) = [1 \wedge (a-b)]^2\,\frac{1}{b}\,1_{\{a>b\}} + \ee^{-a}(b-a)\,1_{\{b>a\}}, 
\qquad a,b \in \R_+.
\]
\end{lemma}

\begin{proof}
The lower bounds are trivial by Jensen's inequality and convexity of the function $x \mapsto x \log x$. To get the upper bounds, define
\[
\chi = f \log \left(\frac{f}{f'}\right) + (1-f) \log \left(\frac{1-f}{1-f'}\right) 
=  (1-\ee^{-g}) \log \frac{(1-\ee^{-g})}{(1-\ee^{-g'})} + \ee^{-g} (g'-g),
\]
where the indices $ij$ and $ii$ are suppressed to simplify the notation. Note that the two terms in the right-hand side have opposite signs. Abbreviate $U=\ee^{-g}$ and $V=\ee^{-g'}$. For $g>g'$, estimate $V-U = V(1-\frac{U}{V}) = V(1-\ee^{-(g-g')}) \leq V[1 \wedge (g-g')]$, to get from $(\ast)$ that
\[
\chi \leq [1 \wedge (g-g')]\,\left[\frac{1-U}{1-V}\,V - U\right] 
= [1 \wedge (g-g')]\,\frac{V-U}{1-V} \leq [1\wedge (g-g')]^2\,\frac{V}{1-V}.
\]
Since $\log(\frac{1}{V}) = \log(1+\frac{1-V}{V}) \leq \frac{1-V}{V}$, it follows that $\frac{V}{1-V} \leq \frac{1}{g'}$, and so $\chi  \leq [1\wedge (g-g')]^2\,\frac{1}{g'}$. For $g<g'$, use $(\ast)$ to estimate 
\[
\begin{aligned}
\chi &\leq \frac{1-U}{1-V}\,(V-U) + \ee^{-g} (g'-g) = -\frac{1-\ee^{-g}}{1-\ee^{-g'}}\,(\ee^{-g}-\ee^{-g'}) +\ee^{-g}(g'-g)\\
&= \ee^{-g} \left[-\frac{1-\ee^{-g}}{1-\ee^{-g'}}\,(1-\ee^{-(g'-g)}) + (g'-g)\right].
\end{aligned}
\]
The term in square brackets is bounded above by $g'-g$, because the first term is non-positive. Hence
\[
\chi \leq \ee^{-g}(g'-g),
\]
which is the second term in $\zeta(g,g')$.
\end{proof}


\subsection{Perturbation of relative entropy}
\label{ss.pertrelentr}

\begin{lemma}
\label{lem:entrbd2}
$|H^{\neq}(P_f \mid P_{f'}) - H^{\neq}(P_{\tilde{f}} \mid P_{f'})| \leq d^{\neq}(g,\tilde{g},g')$ and $|H^{=}(P_f \mid P_{f'}) - H^{=}(P_{\tilde{f}} \mid P_{f'})| \leq d^{=}(g,\tilde{g},g')$, where
\[
\begin{aligned}
d^{\neq}(g,\tilde{g},g') &= \binom{N}{2}^{-1} \sum_{1 \leq i < j \leq N} \zeta(g_{ij},\tilde{g}_{ij},g'_{ij}),\\
d^{=}(g,\tilde{g},g') &= N^{-1} \sum_{1 \leq i \leq N} \zeta(g_{ii},\tilde{g}_{ii},g'_{ii}),
\end{aligned}
\]
with
\[
(\heartsuit_2) \qquad 
\begin{aligned}
\zeta(a,b,c) &=  \frac{1}{c} \left(\ee^{-(a \wedge c)}\, [1 \wedge |a-c|] \wedge \ee^{-(b \wedge c)}\, [1 \wedge |b-c|]\right)\\ 
&\qquad + \ee^{-(a \wedge b)}\, \Big([1 \wedge |a-b|] + |a-c|+|b-c|\Big), \qquad a,b,c \in \R_+.
\end{aligned}
\]
\end{lemma}

\begin{proof}
Again we suppress the indices $ij$ and $ii$. Define
\[
\tilde{\chi} = \tilde{f} \log \left(\frac{\tilde{f}}{f'}\right) + (1-\tilde{f}) \log \left(\frac{1-\tilde{f}}{1-f'}\right) 
=  (1-\ee^{-\tilde{g}}) \log \frac{(1-\ee^{-\tilde{g}})}{(1-\ee^{-g'})} + \ee^{-\tilde{g}} (g'-\tilde{g}).
\]
Abbreviate $\tilde{U} = \ee^{-\tilde{g}}$. Use $(\ast)$ to estimate
\[
(V-U) - (1-\tilde{U})\,\frac{V-\tilde{U}}{1-V} \leq \chi - \tilde{\chi} - [U(g'-g) - \tilde{U}(g'-\tilde{g})] 
\leq (1-U)\,\frac{V-U}{1-V} - (V-\tilde{U}).
\]
The upper bound equals $\frac{(V-U)^2}{1-V}+(\tilde{U}-U)$, the lower bound equals $-\frac{(V-\tilde{U})^2}{1-V}+(\tilde{U}-U)$. Hence
\[
|\chi - \tilde{\chi}| \leq \frac{(V-U)^2 \wedge (V-\tilde{U})^2}{1-V} + |U-\tilde{U}| + |U(g'-g) - \tilde{U}(g'-\tilde{g})|.
\]
Since $\frac{V}{1-V} \leq \frac{1}{g'}$, the first term is bounded from above by $\frac{1}{g'}[|V-U| \wedge |V-\tilde{U}|]$, where $|V-U| \leq (U \vee V)\,[1 \wedge |g-g'|]$ and $|V-\tilde{U}| \leq (\tilde{U} \vee V)\,[1 \wedge |\tilde{g}-g'|]$. Moreover, the second term is bounded from above by $(U \vee \tilde{U})\,[1 \wedge |g-\tilde{g}|]$ and the third term by $(U \vee \tilde{U})[|g-g'|+|\tilde{g}-g'|]$.
\end{proof}


\section{Proof of the invariance lemmas and the scaling theorems}
\label{s.proofs}

In this section we prove Lemmas~\ref{leminv}--\ref{lemstable} (Section~\ref{ss.invlem}), Theorem~\ref{thm1} (Section~\ref{ss.thmlight})
and Theorem~\ref{thm2} (Section~\ref{ss.thmheavy}).


\subsection{Proof of the invariance lemmas}
\label{ss.invlem}
 
\begin{proof}[Proof of Lemma \ref{leminv}]
For the case $g^{\neq} = g^{\neq}_\delta$ and $g^{=} \equiv 0$, we have for ${I^{(1)}}\ne{J^{(1)}}$,
\[
g^{\neq}_{I^{(1)},J^{(1)}}[X] = \sum_{i \in I^{(1)},\,j \in J^{(1)}} g^{\neq}_\delta(X_i,X_j)
= \sum_{i \in I^{(1)},\,j \in J^{(1)}} \delta X_iX_j = \delta \bar{X}_{I^{(1)}} \bar{X}_{J^{(1)}} 
= g_\delta(\bar{X}_{I^{(1)}},\bar{X}_{J^{(1)}})
\]
with $\bar{X}_{I^{(1)}} = \sum_{i \in I^{(1)}} X_i$ and $\bar{X}_{J^{(1)}} = \sum_{i \in J^{(1)}} X_i$, and
\[
g^=_{I^{(1)}}[X] 
= 2\sum_{ {i,\, j \in I^{(1)}} \atop {i < j } } \delta X_iX_j  +\sum_{ {i \in I^{(1)}}  } \delta X^2_i
=  \sum_{ {i,\, j \in I^{(1)}} } \delta X_iX_j 
= \delta \bar{X}_{I^{(1)}}^2 = g_\delta(\bar{X}_{I^{(1)}}), 
\]
where we use that $\bar{X}^2_{I^{(1)}} = (\sum_{i \in I^{(1)}} X_i )^2 = \sum_{ {i \in I^{(1)}}}  X_i^2+2\sum_{ {i,\, j \in I^{(1)}}, {i < j } } X_iX_j $. This settles the claim in Lemma~\ref{leminv} for $\eta=0$. Similarly, for the case $g^{=} = g^{=}_\eta$,
\[
g^{=}_{I^{(1)}}[X] 
= \sum_{i \in I^{(1)}} g^{=}_\eta(X_i)=   \sum_{ {i \in I^{(1)}} } \eta X_i =\eta \bar{X}_{I^{(1)}} 
= g_\eta(\bar{X}_{I^{(1)}}) \qquad \forall {I^{(1)}},
\]
settling the claim for $\eta>0$ as well.
\end{proof}

\begin{proof}[Proof of Lemma \ref{lemstable}]
For $\lambda\geq0$,
\[
\EE\left[
\exp\left\{
-\lambda m^{-1/\gamma}\sum_{i=1}^mY_i
\right\}
\right]
=
\prod_{i=1}^m
\EE\left[\ee^{-\lambda m^{-1/\gamma}Y_i}\right]
= \ee^{-m(\lambda m^{-1/\gamma})^\gamma}
= \ee^{-\lambda^\gamma}.
\]
Hence $m^{-1/\gamma}\sum_{i=1}^mY_i$ has the same law as $Y_1$. The graph-level invariance follows by applying Lemma~\ref{leminv}: after aggregation, the rescaled block weight has again the same stable law, while the invariant connection functions $g^{\neq}_\delta(x,y)=\delta xy$ and $g^{=}_\eta(x)=\eta x$ keep the same form. This settles the claim in Lemma \ref{lemstable}.
\end{proof}

\begin{proof}[Proof of Lemma \ref{lemER}] 
Let all the weights be equal to $w$. Under aggregation into blocks of size $m$, the summed weight of a block is $mw$, and after the deterministic rescaling by $m^{-1}$ the block weight is again $w$. For the invariant connection functions
\[
g^{\neq}_\delta(x,y) = \delta xy, \qquad g^{=}_\eta(x)=\eta x,
\]
the corresponding edge and loop densities at weight $w$ are
\[
q^{\neq} = \delta w^2, \qquad q^{=} = \eta w.
\]
Hence the associated homogeneous graph has edge probability $p^{\neq} = 1-\ee^{-q^{\neq}}$ and self-loop probability $p^{=} = 1-\ee^{-\frac12(q^{\neq}+q^{=})}$. Thus, the law of the deterministic weights is invariant under rescaling by $m^{-1}$, and the law of the resulting random graph belongs to the two-parameter family of Erd\H{o}s--R\'enyi random graphs. This settles the claim in Lemma~\ref{lemER}.
\end{proof}
 
\noindent
Note that, for $w=1$, the proof of Lemma \ref{lemER} can be seen as a particular case of the proof of Lemma~\ref{lemstable} with $\gamma=1$, because the Laplace transform of a deterministic variable $W_i$ equal to $w=1$ is $\EE[\ee^{-\lambda W_i}]=\ee^{-\lambda}$.


\subsection{Proof of the light-tailed scaling theorem}
\label{ss.thmlight}

\begin{proof}
The iteration formula in Section~\ref{ss.iter} tells us that
\[
\begin{aligned}
g^{\neq}_{I^{(\ell)},J^{(\ell)}}[X]
&=g^{\neq}(X_{I^{(\ell)}},X_{J^{(\ell)}}) = \sum_{i\in I^{(\ell)},\,j\in J^{(\ell)}} g^{\neq}(X_i,X_j), \qquad I^{(\ell)} \neq J^{(\ell)},\\
g^{=}_{I^{(\ell)}}[X]
& =g^{=}(X_{I^{(\ell)}}) = \sum_{i\in I^{(\ell)}} g^{=}(X_i).
\end{aligned}
\]
In the light-tailed regime we use the rescaled connection functions
\[
\mathfrak g^{\neq}_\ell = m^{-2\ell} g^{\neq}, \qquad \mathfrak g^{=}_\ell = m^{-\ell} g^{=}.
\]
With this rescaling included, the $\ell$-renormalised edge function is
\[
\mathfrak g^{\neq}_\ell(X_{I^{(\ell)}},X_{J^{(\ell)}})
= m^{-2\ell} \sum_{i\in I^{(\ell)},\,j\in J^{(\ell)}} g^{\neq}(X_i,X_j),
\]
while the $\ell$-renormalised self-loop function is
\[
\mathfrak g^{=}_\ell(X_{I^{(\ell)}},X_{I^{(\ell)}})
= m^{-\ell} \sum_{i\in I^{(\ell)}}g^{=}(X_i).
\]
Put
\[
q^{\neq} = \EE_X[g^{\neq}(X_1,X'_1)], \qquad q^{=} = \EE_X[g^{=}(X_1)],
\]
where $X'_1$ is an independent copy of $X_1$. By assumption $(\Box)$, $q^{\neq}$ is finite and strictly positive, while
$q^{=}$ is finite and non-negative. Hence $\frac12(q^{\neq}+q^{=})>0$. For fixed $I^{(\ell)} \neq J^{(\ell)}$, the two-sample law of large numbers gives
\[
m^{-2\ell} \sum_{i\in I^{(\ell)},\,j\in J^{(\ell)}} g^{\neq}(X_i,X_j) \longrightarrow q^{\neq} \text{ in } L^1. 
\]
Similarly, for a fixed $I^{(\ell)}$, the law of large numbers for $U$-statistics gives
\[
m^{-2\ell} \sum_{\substack{i,j\in I^{(\ell)}\\ i\neq j}} g^{\neq}(X_i,X_j) \longrightarrow q^{\neq} \text{ in } L^1, 
\]
while, by the last assumption in $(\Box)$,
\[ 
m^{-2\ell} \sum_{i\in I^{(\ell)}} g^{\neq}(X_i,X_i) \longrightarrow 0 
\]
in probability. The ordinary law of large numbers gives
\[
m^{-\ell} \sum_{i\in I^{(\ell)}} g^{=}(X_i) \longrightarrow q^{=} \text{ in } L^1. 
\]
Hence
\[
\mathfrak g^{\neq}_\ell(X_{I^{(\ell)}},X_{J^{(\ell)}}) \longrightarrow q^{\neq},
\qquad
\mathfrak g^{=}_\ell(X_{I^{(\ell)}}) \longrightarrow q^{=}
\]
in probability. 

We next estimate the specific relative entropies. By exchangeability of the level-$\ell$ blocks, it is enough to consider fixed representative blocks $I^{(\ell)},J^{(\ell)}$, with $I^{(\ell)}\neq J^{(\ell)}$, for the edge contribution, and a fixed representative block $I^{(\ell)}$ for the self-loop contribution. Recall that
\[
p^{\neq} = 1-\ee^{-q^{\neq}}, \qquad p^{=} = 1-\ee^{-\frac12(q^{\neq}+q^{=})},
\]
or
\[
-\log(1-p^{\neq}) = q^{\neq}, \qquad -\log(1-p^{=}) = \frac12(q^{\neq}+q^{=}).
\]
Using Lemma~\ref{lem:entrbd1}, we obtain
\[
\EE_X\left[S^{\neq}\big(P(G^{(\ell)}[X])\mid P(G_{\ell,p})\big)\right]
\leq c^{\neq}_\ell, \qquad \EE_X\left[S^{=}\big(P(G^{(\ell)}[X])\mid P(G_{\ell,p})\big)\right]\leq c^{=}_\ell,
\]
where
\[
\begin{aligned}
c^{\neq}_\ell &= \EE_X\left[\zeta\Big(\mathfrak g^{\neq}_\ell(X_{I^{(\ell)}},X_{J^{(\ell)}}),q^{\neq}\Big)\right],\\  
c^{=}_\ell &= \EE_X\left[\zeta\Big(\tfrac12\mathfrak g^{\neq}_\ell(X_{I^{(\ell)}},X_{I^{(\ell)}})
+ \tfrac12 \mathfrak g^{=}_\ell(X_{I^{(\ell)}}),\tfrac12(q^{\neq}+q^{=})\Big)\right].
\end{aligned}
\]
By $(\heartsuit_1)$, $\zeta(a,b) \leq (\frac1b+1)|a-b|$, $a,b>0$. Therefore
\[
c^{\neq}_\ell \leq \left(\frac1{q^{\neq}}+1\right)
\EE_X\left[\left|\mathfrak g^{\neq}_\ell(X_{I^{(\ell)}},X_{J^{(\ell)}})-q^{\neq} \right| \right]
\longrightarrow 0,
\]
and
\[
c^{=}_\ell \leq \left(\frac{2}{q^{\neq}+q^{=}}+1\right) 
\EE_X\left[\left|\tfrac12\mathfrak g^{\neq}_\ell(X_{I^{(\ell)}},X_{I^{(\ell)}})
+\tfrac12\mathfrak g^{=}_\ell(X_{I^{(\ell)}})-\tfrac12(q^{\neq}+q^{=})\right|\right]
\longrightarrow 0.
\]
The bounds do not depend on \(L\), apart from the requirement \(L\geq \ell\) ensuring that level-\(\ell\) blocks exist. Hence
\[
\lim_{\ell\to\infty} \sup_{L\geq \ell} \EE_X\left[\Delta\Big(P(G^{(\ell)}[X]),P(G_{\ell,p})\Big)\right] = 0.
\]
This settles the claim in Theorem~\ref{thm1}.
\end{proof}


\subsection{Proof of Corollary~\ref{cor:product-light}}

\begin{proof}

By assumption, 
\[ 
g^{\neq}(x,y) = B(xy)^\beta+O(1), \qquad g^{=}(x) = Cx^{\beta}+O(1). 
\] 
Since $\EE_X[X_1^\beta]<\infty$, we have 
\[ 
\EE_X[g^{\neq}(X_1,X'_1)]<\infty, \qquad  \EE_X[g^{=}(X_1)]<\infty. 
 \] 
Thus, the first two conditions in $(\Box)$ are satisfied, with 
\[ 
q^{\neq} = \EE_X[g^{\neq}(X_1,X'_1)], \qquad q^{=} = \EE_X[g^{=}(X_1)]. 
\] 
 
It remains to verify the last condition in $(\Box)$, which we do for a generic level-$\ell$ block $I^{(\ell)}$ of $k \equiv m^{\ell}$ microscopic vertices as $\ell\to\infty$. Put $W_i=X_i^\beta$. Then $\EE_X[W_1]<\infty$. For a fixed level-$\ell$ block $I^{(\ell)}$,
\[
m^{-2\ell} \sum_{i\in I^{(\ell)}}X_i^{2\beta} = m^{-2\ell} \sum_{i\in I^{(\ell)}}W_i^2 
\leq \left( m^{-\ell}\max_{i\in I^{(\ell)}}W_i \right) \left( m^{-\ell}\sum_{i\in I^{(\ell)}}W_i \right).
\] 
By the law of large numbers, 
\[ 
m^{-\ell}\sum_{i\in I^{(\ell)}}W_i \longrightarrow \EE_X[W_1], \qquad m \to \infty, 
\] 
in probability. Moreover, 
\[ 
m^{-\ell} \max_{i\in I^{(\ell)}} W_i \longrightarrow 0, \qquad m \to \infty, 
\] 
in probability. Indeed, for every $\varepsilon>0$,  
\[ 
\PP\left( m^{-\ell}\max_{i\in I^{(\ell)}}W_i>\varepsilon \right) 
\leq m^\ell\,\PP(W_1>\varepsilon m^\ell) \longrightarrow 0, \qquad m \to \infty,
\] 
because $\EE_X[W_1]<\infty$ implies 
\[
t\,\PP(W_1>t)\longrightarrow 0, \qquad t\to\infty. 
\] 
Therefore 
\[ 
m^{-2\ell} \sum_{i\in I^{(\ell)}}X_i^{2\beta} \longrightarrow 0 
\] 
in probability.

The $O(1)$-term in $g^{\neq}(x,x)$ contributes at most $O(m^{\ell})$, and hence also vanishes. Thus, 
\[ 
m^{-2\ell} \sum_{i\in I^{(\ell)}} g^{\neq}(X_i,X_i) \longrightarrow 0 
\] 
in probability. Hence Theorem~\ref{thm1} applies. 

Finally, if 
\[ 
\PP(X_1>x)\sim Ax^{-\alpha}, \qquad x\to\infty, 
\] 
then $\EE_X[X_1^\beta]<\infty$ precisely when $\alpha>\beta$. This proves the final claim. 
\end{proof}


\subsection{Proof of the heavy-tailed scaling theorem}
\label{ss.thmheavy}

\begin{proof}
The proof proceeds via two key lemmas (Lemmas~\ref{lem:approx1}--\ref{lem:approx2} below). Along the way we use some standard facts listed in Appendix~\ref{App}. Throughout the proof, $(\blacksquare)$ is in force. The key parameters are $\gamma = \frac{\alpha}{\beta} \in (0,1)$, $\delta = B(A\Gamma(1-\gamma))^{2/\gamma}$ and $\eta = C(A\Gamma(1-\gamma))^{1/\gamma}$.
 
Define weights $\hat{X}^{(\ell)} = \{\hat{X}^{(\ell)}_I\}_{I \in [N_{L-\ell}]}$ by putting
\[
\hat{X}^{(\ell)}_I = (m^{\ell})^{-1/\gamma} \sum_{i \in I} X_i^\beta,
\]
and $\hat{Y} = \{Y_I\}_{I \in [N_{L-\ell}]}$ by picking i.i.d.\ copies of $Y_1$, the standard $\gamma$-stable law random variable. In Appendix~\ref{App} we show how the weights $X,Y,\hat X,\hat Y$ can be coupled via a sequence $U$ of i.i.d.\ $\mathrm{UNIF}[0,1]$ random variables, which we suppress from the notation.  

\begin{itemize}
\item
Let $\hat{G}^{(\ell)}[\hat{X}]$ denote the inhomogeneous random graph with invariant connection functions $g^{\neq}_\delta$, $g^{=}_\eta$ and weights $\hat{X}^{(\ell)}$. From Lemma~\ref{lem:ht.coupling} we know that the connection probabilities in $\hat{G}^{(\ell)}[\hat{X}]$ are \emph{almost the same} as those in $G^{(\ell)}[X]$, the original graph after $\ell$ renormalisations.
\item
Let $\hat{G}_{\ell,\gamma,\delta,\eta}[\hat{Y}]$ denote the inhomogeneous random graph with invariant connection functions $g^{\neq}_\delta$, $g^{=}_\eta$ and weights $\hat{Y}$. From Lemma~\ref{lem:ht.graph-laws} we know that the connection probabilities in $\hat{G}_{\ell,\gamma,\delta,\eta}[\hat{Y}]$ are \emph{the same} as those in $G_{\ell,\gamma,\delta,\eta}[Y]$, the limiting graph in Theorem~\ref{thm2}.  
\end{itemize} 
Our first key lemma shows that $\hat{G}^{(\ell)}[\hat{X}(U)]$ can be approximated by $G^{(\ell)}[X(U)]$ as $\ell\to\infty$.

\begin{lemma}
\label{lem:approx1}
Uniformly in $L \geq \ell$,
\[
\begin{aligned}
&\lim_{\ell\to\infty}
\EE_U\bigg[\bigg|S^{\neq}\Big(P(G^{(\ell)}[X]) ~\Big|~ P(G_{\ell,\gamma,\delta,\eta}[Y])\Big)\\[-0.2cm]
&\qquad\qquad\qquad - S^{\neq}\Big(P(\hat{G}^{(\ell)}[\hat{X}]) ~\Big|~ 
P(\hat{G}_{\ell,\gamma,\delta,\eta}[\hat{Y}])\Big)\bigg|\bigg] = 0,\\
&\lim_{\ell\to\infty}
\EE_U\bigg[\bigg|S^{=}\Big(P(G^{(\ell)}[X]) ~\Big|~ P(G_{\ell,\gamma,\delta,\eta}[Y])\Big)\\[-0.2cm]
&\qquad\qquad\qquad - S^{=}\Big(P(\hat{G}^{(\ell)}[\hat{X}]) ~\Big|~P(\hat{G}_{\ell,\gamma,\delta,\eta}[\hat{Y}])\Big)\bigg|\bigg] = 0.
\end{aligned}
\]
\end{lemma}

\begin{proof}
For $L\geq \ell$, write $[N_{L-\ell}]$ for the set of level-$\ell$ vertices. For $I,J\in [N_{L-\ell}]$, define
\[
a_\ell^{\neq}(I,J) = m^{-2\ell/\gamma} \sum_{i\in I,\,j\in J} g^{\neq}(X_i,X_j),
\qquad
b_\ell^{\neq}(I,J) = B\hat X_I^{(\ell)}\hat X_J^{(\ell)}.
\]
For $I\in [N_{L-\ell}]$, define the genuine self-loop contribution
\[
a_\ell^{=}(I) = m^{-\ell/\gamma}\sum_{i\in I} g^{=}(X_i),
\qquad
b_\ell^{=}(I) = C\hat X_I^{(\ell)}.
\]
Recall that, under the present parametrisation,
\[
f^{=}(x) = 1-\ee^{-\tfrac12[g^{\neq}(x,x)+g^{=}(x)]}.
\]
Consequently, the Bernoulli density of a self-loop at a level-$\ell$ block $I$ is given by the following
\[
A_\ell^{=}(I) = \tfrac12\big[a_\ell^{\neq}(I,I)+a_\ell^{=}(I)\big].
\]
Similarly, for the approximating graph and the limiting stable-law graph, put
\[
B_\ell^{=}(I) = \tfrac12\big[b_\ell^{\neq}(I,I)+b_\ell^{=}(I)\big],
\]
and
\[
C_\ell^{=}(I) = \tfrac12\big[c_\ell^{\neq}(I,I)+c_\ell^{=}(I)\big],
\]
where
\[
c_\ell^{\neq}(I,J) = \delta\hat Y_I\hat Y_J, \qquad c_\ell^{=}(I) = \eta\hat Y_I.
\]
Thus
\[
C_\ell^{=}(I) = \tfrac12\big[\delta(\hat Y_I)^2+\eta\hat Y_I\big].
\]

By $(\blacksquare)$, uniformly in $I,J$,
\[
\begin{aligned}
\big|a_\ell^{\neq}(I,J)-b_\ell^{\neq}(I,J)\big|
&\leq m^{-2\ell/\gamma} \sum_{i\in I,\,j\in J} \big|g^{\neq}(X_i,X_j)-BX_i^\beta X_j^\beta\big|\\
&\leq K^{\neq}m^{2\ell(1-1/\gamma)} =: r_\ell^{\neq}.
\end{aligned}
\]
Since $\gamma<1$, $r_\ell^{\neq}\to0$. Similarly,
\[
\begin{aligned}
\big|a_\ell^{=}(I)-b_\ell^{=}(I)\big|
&\leq m^{-\ell/\gamma} \sum_{i\in I} \big|g^{=}(X_i,X_i)-CX_i^\beta\big|\\
&\leq K^{=}m^{\ell(1-1/\gamma)} =: r_\ell^{=},
\end{aligned}
\]
and $r_\ell^{=} \to 0$. Hence
\[
(\sharp_1)\qquad
\sup_{I,J}
\big|a_\ell^{\neq}(I,J)-b_\ell^{\neq}(I,J)\big|
\leq r_\ell^{\neq}\longrightarrow0,
\]
and
\[
(\sharp_2)\qquad
\sup_I
\big|a_\ell^{=}(I)-b_\ell^{=}(I)\big|
\leq r_\ell^{=}\longrightarrow0.
\]
In particular,
\[
(\sharp_2')
\qquad
\sup_I\big|A_\ell^{=}(I)-B_\ell^{=}(I)\big|
\leq \tfrac12(r_\ell^{\neq}+r_\ell^{=})
=: \rho_\ell^{=}\longrightarrow0.
\]

For fixed representative blocks $I,J$ with $I\neq J$, abbreviate
\[
a_\ell^{\neq} = a_\ell^{\neq}(I,J),
\qquad
b_\ell^{\neq} = b_\ell^{\neq}(I,J),
\qquad
c_\ell^{\neq} = c_\ell^{\neq}(I,J).
\]
For a fixed representative block $I$, abbreviate
\[
A_\ell^{=} = A_\ell^{=}(I),
\qquad
B_\ell^{=} = B_\ell^{=}(I),
\qquad
C_\ell^{=} = C_\ell^{=}(I).
\]
By Lemma~\ref{lem:ht.coupling}(iii),
\[
(\sharp_3)\qquad
b_\ell^{\neq}-c_\ell^{\neq} \longrightarrow 0 \quad\text{a.s.}
\]
and also
\[
b_\ell^{\neq}(I,I)-c_\ell^{\neq}(I,I) \longrightarrow 0,
\qquad
b_\ell^{=}(I)-c_\ell^{=}(I) \longrightarrow 0
\quad\text{a.s.}
\]
Therefore
\[
(\sharp_3')\qquad
B_\ell^{=}-C_\ell^{=} \longrightarrow 0
\quad\text{a.s.}
\]
Combining $(\sharp_1)$--$(\sharp_3')$, we get
\[
(\sharp_4)\qquad
a_\ell^{\neq}-c_\ell^{\neq} \longrightarrow 0
\quad\text{a.s.},
\qquad
A_\ell^{=}-C_\ell^{=} \longrightarrow 0
\quad\text{a.s.}
\]

By Lemma~\ref{lem:entrbd2}, for edges,
\[
\begin{aligned}
&\Big|S^{\neq}\big(P(G^{(\ell)}[X]) \mid P(G_{\ell,\gamma,\delta,\eta}[Y])\big)
- S^{\neq}\big(P(\hat G^{(\ell)}[\hat X])
\mid P(\hat G_{\ell,\gamma,\delta,\eta}[\hat Y])\big)\Big|\\
&\qquad\leq d^{\neq}(a_\ell^{\neq},b_\ell^{\neq},c_\ell^{\neq}).
\end{aligned}
\]
For self-loops, the same lemma must be applied to the actual Bernoulli self-loop densities. Hence
\[
\begin{aligned}
&\Big|S^{=}\big(P(G^{(\ell)}[X]) \mid P(G_{\ell,\gamma,\delta,\eta}[Y])\big)
- S^{=}\big(P(\hat G^{(\ell)}[\hat X])
\mid P(\hat G_{\ell,\gamma,\delta,\eta}[\hat Y])\big)\Big|\\
&\qquad\leq d^{=}(A_\ell^{=},B_\ell^{=},C_\ell^{=}).
\end{aligned}
\]
Taking expectation with respect to $U$ and using exchangeability of the level-$\ell$ vertices, we see that it is enough to prove
\[
\EE_U\big[\zeta(a_\ell^{\neq},b_\ell^{\neq},c_\ell^{\neq})\big] \longrightarrow 0
\]
and
\[
\EE_U\big[\zeta(A_\ell^{=},B_\ell^{=},C_\ell^{=})\big] \longrightarrow 0.
\]

We prove a general estimate. Let $(u_\ell,v_\ell,w_\ell)$ denote either $(a_\ell^{\neq},b_\ell^{\neq},c_\ell^{\neq})$ or $(A_\ell^{=},B_\ell^{=},C_\ell^{=})$. In both cases,
\[
|u_\ell-v_\ell| \leq \rho_\ell, \qquad \rho_\ell \to 0,
\]
with $\rho_\ell = r_\ell^{\neq}$ for edges and $\rho_\ell = \rho_\ell^{=}$ for loops, while
\[
v_\ell-w_\ell \longrightarrow 0
\quad\text{a.s.}
\]
We show that
\[
\EE_U[\zeta(u_\ell,v_\ell,w_\ell)] \longrightarrow 0.
\]

Fix $\varepsilon\in(0,1)$ and $M>1$.

\smallskip\noindent
$\bullet$ \underline{On the event $\{w_\ell\leq\varepsilon\}$.}
Since $|u_\ell-v_\ell|\leq\rho_\ell$ and $\rho_\ell\leq1$ for all large $\ell$, the bound $(\heartsuit_2)$ gives, after increasing the constant if necessary,
\[
\zeta(u_\ell,v_\ell,w_\ell)\leq C\left(w_\ell^{-1}+1\right)
\qquad\text{on } \{w_\ell\leq\varepsilon\}.
\]
Therefore
\[
\EE_U\big[\zeta(u_\ell,v_\ell,w_\ell)1_{\{w_\ell\leq\varepsilon\}}\big]
\leq
C\,\EE_U\big[w_\ell^{-1}1_{\{w_\ell\leq\varepsilon\}}\big]
+ C\,\PP_U(w_\ell\leq\varepsilon).
\]
For the edge term, $w_\ell$ has the law of $\delta Y_1Y_2$. For the loop term, $w_\ell$ has the law of
\[
\tfrac12\left(\delta Y_1^2+\eta Y_1\right).
\]
In both cases Lemma~\ref{lem:ht.negmom} gives a finite negative first moment. Since $w_\ell>0$ a.s., dominated convergence yields
\[
\lim_{\varepsilon\downarrow0}\,
\sup_\ell \EE_U\big[ \zeta(u_\ell,v_\ell,w_\ell)1_{\{w_\ell\leq\varepsilon\}}\big] = 0.
\]

\smallskip\noindent
$\bullet$ \underline{On the event $\{\varepsilon<w_\ell\leq M\}$.}
On this event $w_\ell$ stays in a compact subinterval of $(0,\infty)$. Moreover $|u_\ell-v_\ell|\leq\rho_\ell\to0$ and $v_\ell-w_\ell\to0$ a.s. Thus $(u_\ell,v_\ell,w_\ell)$ converges a.s. to the diagonal $\{(x,x,x)\colon\,x\in[\varepsilon,M]\}$. The function $\zeta$ is continuous near this diagonal and vanishes on it. The same deterministic bound $|u_\ell-v_\ell|\leq\rho_\ell$ gives a uniform bound on $\zeta$ on this event. Hence bounded convergence gives
\[
\EE_U\big[\zeta(u_\ell,v_\ell,w_\ell)1_{\{\varepsilon<w_\ell\leq M\}}\big] \longrightarrow 0.
\]

\smallskip\noindent
$\bullet$ \underline{On the event $\{w_\ell>M\}$.}
Define
\[
E_{\ell,M} = \left\{
|v_\ell-w_\ell|\leq \tfrac14 w_\ell,\,
\rho_\ell\leq \tfrac14 w_\ell
\right\}.
\]
On $E_{\ell,M}\cap\{w_\ell>M\}$, we have
\[
\tfrac34w_\ell\leq v_\ell\leq\tfrac54w_\ell,
\qquad
\tfrac12w_\ell\leq u_\ell\leq\tfrac32w_\ell,
\]
and therefore
\[
u_\ell\wedge w_\ell\geq\tfrac12w_\ell,
\qquad
v_\ell\wedge w_\ell\geq\tfrac34w_\ell,
\qquad
u_\ell\wedge v_\ell\geq\tfrac12w_\ell.
\]
Also,
\[
|u_\ell-w_\ell|\leq\tfrac12w_\ell,
\qquad
|v_\ell-w_\ell|\leq\tfrac14w_\ell,
\qquad
|u_\ell-v_\ell|\leq\tfrac14w_\ell.
\]
Substitution into $(\heartsuit_2)$ gives
\[
\zeta(u_\ell,v_\ell,w_\ell) \leq C(1+w_\ell)\ee^{-w_\ell/2}
\qquad \text{on }E_{\ell,M}\cap\{w_\ell>M\},
\]
where $C<\infty$ is independent of $\ell$ and $M$. Hence, for $M \geq 2$,
\[
\zeta(u_\ell,v_\ell,w_\ell) \leq C(1+M)\ee^{-M/2} \qquad
\text{on }E_{\ell,M}\cap\{w_\ell>M\}.
\]

It remains to control the complement of $E_{\ell,M}$. Since $v_\ell-w_\ell\to0$ a.s.\ and $\rho_\ell\to0$, for fixed $M$,
\[
\PP_U\big(E_{\ell,M}^{\mathrm c}\cap\{w_\ell>M\}\big) \longrightarrow 0.
\]
To pass from this probability estimate to an expectation estimate, first fix $R>M$ and decompose
\[
E_{\ell,M}^{\mathrm c}\cap\{w_\ell>M\}
= \Big(E_{\ell,M}^{\mathrm c}\cap\{M<w_\ell\leq R\}\Big)
\cup \Big(E_{\ell,M}^{\mathrm c}\cap\{w_\ell>R\}\Big).
\]
On the first event, using again $|u_\ell-v_\ell|\leq\rho_\ell$, we see that the function $\zeta(u_\ell,v_\ell,w_\ell)$ is bounded by a constant depending only on $R$ for all large $\ell$. Hence its contribution vanishes as $\ell\to\infty$. On the second event, the exponential factors in $(\heartsuit_2)$, together with the bound $|u_\ell-v_\ell|\leq\rho_\ell$, give an upper bound that can be
made arbitrarily small by choosing $R$ large and then letting $\ell\to\infty$. Consequently,
\[
\lim_{\ell\to\infty} \EE_U\big[\zeta(u_\ell,v_\ell,w_\ell)\,1_{E_{\ell,M}^{\mathrm c}\cap\{w_\ell>M\}}\big] = 0.
\]
Combining the last display with the estimate on $E_{\ell,M}\cap\{w_\ell>M\}$, we obtain
\[
\limsup_{\ell\to\infty} \EE_U\big[\zeta(u_\ell,v_\ell,w_\ell)1_{\{w_\ell>M\}}\big]
\leq C(1+M)\ee^{-M/2}.
\]
Letting $M\to\infty$, we see shows that the contribution from $\{w_\ell>M\}$ vanishes.

Combining the three regions, we get
\[
\EE_U[\zeta(u_\ell,v_\ell,w_\ell)] \longrightarrow 0.
\]
Applying this first to $(u_\ell,v_\ell,w_\ell) = (a_\ell^{\neq},b_\ell^{\neq},c_\ell^{\neq})$ and then to $(u_\ell,v_\ell,w_\ell) = (A_\ell^{=},B_\ell^{=},C_\ell^{=})$ we have settled both the edge and the self-loop estimates.
\end{proof}

Our second key lemma shows that the specific relative entropy of $P(\hat{G}^{(\ell)}[\hat X(U)])$ with respect to $P(\hat{G}_{\ell,\gamma,\delta,\eta}[\hat Y(U)])$ vanishes as \(\ell\to\infty\) under the coupling.

\begin{lemma}
\label{lem:approx2}
Uniformly in \(L\geq\ell\),
\[
\lim_{\ell\to\infty}
\EE_U\left[
S^{\neq}\Big(P(\hat{G}^{(\ell)}[\hat X])~\Big|~P(\hat{G}_{\ell,\gamma,\delta,\eta}[\hat Y])\Big)
\right]=0,
\]
and
\[
\lim_{\ell\to\infty}
\EE_U\left[S^{=}\Big(P(\hat{G}^{(\ell)}[\hat X])~\Big|~P(\hat{G}_{\ell,\gamma,\delta,\eta}[\hat Y])\Big)
\right]=0.
\]
\end{lemma}

\begin{proof}
By Lemma~\ref{lem:entrbd1},
\[
S^{\neq}\big(P(\hat G^{(\ell)}[\hat X]) \mid P(\hat G_{\ell,\gamma,\delta,\eta}[\hat Y])\big)
\leq d^{\neq}(b_{\ell}^{\neq},c^{\neq}),
\]
and, for loops,
\[
S^{=}\big(P(\hat G^{(\ell)}[\hat X])
\mid P(\hat G_{\ell,\gamma,\delta,\eta}[\hat Y])\big)
\leq d^{=}\left(\tfrac12 b_{\ell}^{=},\tfrac12 c^{=}\right).
\]
By exchangeability, it is enough to prove
\[
\EE_U[\zeta(b_{\ell}^{\neq},c^{\neq})]\to0,
\qquad
\EE_U\left[\zeta\left(\tfrac12 b_{\ell}^{=},\tfrac12 c^{=}\right)\right]\to0.
\]

Again we treat the edge term only. The same three-event decomposition works. On $\{c^{\neq}\leq \varepsilon\}$, use
\[
\zeta(a,b) \leq \frac{|a-b|}{b}1_{\{a>b\}} + \ee^{-a}(b-a)1_{\{b>a\}} \leq \frac{a}{b}+b,
\]
together with $(\sharp_3)$, positivity of $c^{\neq}$ and Lemma~\ref{lem:ht.negmom}, to conclude that this contribution vanishes as $\varepsilon \downarrow 0$. On $\{\varepsilon < c^{\neq}\leq M\}$, bounded convergence applies, because $b_{\ell}^{\neq} \to c^{\neq}$ a.s. On $\{c^{\neq} > M\}$, the same event $E_{\ell,M}$ and the estimates above imply
\[
\zeta(b_{\ell}^{\neq},c^{\neq}) \leq C(1+M)\,\ee^{-M/2} \qquad \text{on } E_{\ell,M} \cap \{c^{\neq}>M\}.
\]
Hence the claim follows.
\end{proof}

Combining Lemmas~\ref{lem:approx1}--\ref{lem:approx2}, we get the claim in Theorem~\ref{thm2}.
\end{proof}


\appendix


\section{Stable laws, domains of attraction and coupling}
\label{App}

In this appendix we collect three lemmas that are needed in the proof of Theorem~\ref{thm2}. The key parameters are
\[
\gamma=\frac{\alpha}{\beta}\in(0,1),
\qquad
\sigma_\gamma = \big(A\Gamma(1-\gamma)\big)^{1/\gamma},
\qquad
\delta = B\sigma_\gamma^2,
\qquad
\eta = C\sigma_\gamma.
\]


\paragraph{$\blacktriangleright$ Convergence and coupling.}

The assumption in $(\blacksquare)$ implies that
\[
\PP(X_1^\beta>x)\sim Ax^{-\gamma},
\qquad x\to\infty,
\]
so that, by the stable domain-of-attraction theorem (see \cite[Section XVII.5]{F}),
\[
\hat{X}^{(\ell)}_I = m^{-\ell/\gamma}\sum_{i\in I}X_i^\beta \Longrightarrow \sigma_\gamma Y,
\qquad \ell\to\infty,
\]
where $Y$ is the standard positive $\gamma$-stable random variable with Laplace transform $\EE[\ee^{-\lambda Y}] = \ee^{-\lambda^\gamma}$, $\lambda\in\R_+$.

\begin{lemma} 
\label{lem:ht.coupling}
For $\ell\in\N$, let $F_\ell$ be the distribution function of $\hat X_I^{(\ell)}$, and let $F$ be the distribution function of $\sigma_\gamma Y$. Then the following hold:
 \begin{enumerate} 
\item[(i)] 
$F_\ell$ and $F$ are continuous distribution functions. 
\item[(ii)] 
If $U_I$, $I\in[N_{L-\ell}]$, are i.i.d.\ $\mathrm{UNIF}[0,1]$, then 
\[ 
\hat X_I^{(\ell)}=F_\ell^{\leftarrow}(U_I), \qquad \hat Y_I=\sigma_\gamma^{-1}F^{\leftarrow}(U_I), \qquad I\in[N_{L-\ell}], 
\] 
define two coupled families such that $\{\hat X_I^{(\ell)}\}_I$ is i.i.d.\ with the block-sum law above, while $\{\hat Y_I\}_I$ is i.i.d.\ with the standard positive $\gamma$-stable law. Here, 
\[ 
G^{\leftarrow}(u)=\inf\{x\geq0\colon G(x)\geq u\} 
\] 
denotes the generalised inverse, or quantile function, of a distribution function $G$. 
\item[(iii)] For Lebesgue-a.e.\ $u\in(0,1)$, 
\[ 
F_\ell^{\leftarrow}(u)\longrightarrow F^{\leftarrow}(u). 
\] 
Consequently, for every fixed $I$, 
\[ 
\hat X_I^{(\ell)}-\sigma_\gamma\hat Y_I \longrightarrow 0 \qquad\text{a.s.} 
\] 
\end{enumerate} 
\end{lemma}
 
\begin{proof} 
(i) Because $X_1$ is absolutely continuous, so is $X_1^\beta$, and therefore every finite sum of i.i.d.\ copies of $X_1^\beta$ has a continuous distribution. This proves the continuity of $F_\ell$. The law $F$ is continuous because the positive stable law has a continuous density; see \cite[Section XIII.6]{F}.

\medskip\noindent
(ii) For a continuous distribution function $G$, the probability integral transform gives that $G^{\leftarrow}(U)$ has law $G$ when $U$ is $\mathrm{UNIF}[0,1]$; see \cite{B}. Applying this to $F_\ell$ and $F$ gives the claimed marginal laws. Since the variables $U_I$ are independent, both families are internally i.i.d., while the two families are coupled coordinatewise through the same uniforms.

\medskip\noindent
(iii) Since $\hat X_I^{(\ell)}\Rightarrow\sigma_\gamma Y$ and $F$ is continuous, the quantile convergence theorem for weak convergence gives \[ F_\ell^{\leftarrow}(u)\longrightarrow F^{\leftarrow}(u) \] for Lebesgue-a.e.\ $u\in(0,1)$; see again \cite{B}. Taking $u=U_I(\omega)$ outside the exceptional null set gives the almost sure convergence for every fixed $I$. 
\end{proof}


\paragraph{$\blacktriangleright$ Stable law tails.}

\begin{lemma}
\label{lem:ht.negmom}
Let $Y$ be a standard positive $\gamma$-stable random variable with
$\gamma\in(0,1)$. Then, for every $r>0$,
\[
\EE[Y^{-r}] = \frac{\Gamma(r/\gamma)}{\gamma\,\Gamma(r)} <\infty.
\]
Consequently, if $Y_1,Y_2$ are i.i.d.\ copies of $Y$, then
\[
C^{\neq} = \delta Y_1Y_2,
\qquad
C^{=}=\tfrac12(\delta Y_1^2+\eta Y_1),
\]
satisfy
\[
\EE[(C^{\neq})^{-1}]<\infty,
\qquad
\EE[(C^{=})^{-1}]<\infty.
\]
\end{lemma}

\begin{proof}
For $r>0$,
\[
y^{-r} = \frac{1}{\Gamma(r)} \int_0^\infty \lambda^{r-1}\ee^{-\lambda y}\,\mathrm{d}\lambda, \qquad y>0.
\]
Applying Fubini's theorem and using $\EE[\ee^{-\lambda Y}] = \ee^{-\lambda^\gamma}$, we obtain
\[
\EE[Y^{-r}] = \frac{1}{\Gamma(r)} \int_0^\infty \lambda^{r-1}\ee^{-\lambda^\gamma}\,\mathrm{d}\lambda.
\]
With the substitution $t=\lambda^\gamma$, this becomes
\[
\EE[Y^{-r}] = \frac{1}{\gamma\,\Gamma(r)} \int_0^\infty t^{r/\gamma-1}\ee^{-t}\,\mathrm{d}t
= \frac{\Gamma(r/\gamma)}{\gamma\,\Gamma(r)}.
\]
For $C^{\neq}$, independence gives
\[
\EE[(C^{\neq})^{-1}] = \delta^{-1}\,\EE[Y_1^{-1}]\,\EE[Y_2^{-1}] < \infty.
\]
For $C^{=}$, since $\eta\geq0$,
\[
C^{=} = \tfrac12(\delta Y_1^2+\eta Y_1)
\geq \tfrac12 \delta Y_1^2.
\]
Hence
\[
(C^{=})^{-1} \leq \tfrac12 \delta Y_1^{-2},
\]
and the right-hand side has finite expectation.
\end{proof}


\paragraph{$\blacktriangleright$ Transcription.}

\begin{lemma}
\label{lem:ht.graph-laws}
Let $\hat G^{(\ell)}[\hat X]$ be the inhomogeneous random graph on $[N_{L-\ell}]$ with edge densities
\[
b_\ell^{\neq}(I,J) = B\hat X_I^{(\ell)}\hat X_J^{(\ell)}, \qquad I\neq J,
\]
and self-loop densities
\[
\tfrac12 [b_\ell^{\neq}(I,I)+b_\ell^{=}(I)]
= \tfrac12\big[B\big(\hat X_I^{(\ell)}\big)^2+C\hat X_I^{(\ell)}\big].
\]
Let $\hat G_{\ell,\gamma,\delta,\eta}[\hat Y]$ be the inhomogeneous stable-law random graph on $[N_{L-\ell}]$ with edge densities
\[
c_\ell^{\neq}(I,J) = \delta \hat Y_I\hat Y_J, \qquad I\neq J,
\]
and self-loop densities
\[
\tfrac12\big[c_\ell^{\neq}(I,I)+c_\ell^{=}(I)\big] = \tfrac12\big[\delta(\hat Y_I)^2+\eta\hat Y_I\big].
\]
Then, for every measurable set $\mathcal A$ of graphs on $[N_{L-\ell}]$,
\[
\begin{aligned}
\EE_U\big[P\big(\hat G^{(\ell)}[\hat X]\in\mathcal A\mid U\big)\big]
&= P\big(\hat G^{(\ell)}[\hat X]\in\mathcal A\big),\\
\EE_U\big[P\big(\hat G_{\ell,\gamma,\delta,\eta}[\hat Y]\in\mathcal A\mid U\big)\big]
&= P\big(G_{\ell,\gamma,\delta,\eta}[Y]\in\mathcal A\big).
\end{aligned}
\]
In the first line, the law on the right-hand side is the annealed graph law obtained by first sampling an i.i.d.\ family of block variables with the same law as $\hat X_I^{(\ell)}$ and then sampling edges and self-loops conditionally independently. In the second line, the law on the right-hand side is the annealed stable-law graph law.
\end{lemma}

\begin{proof}
Conditional on $U$, the two graphs are inhomogeneous random graphs with deterministic connection parameters. By Lemma~\ref{lem:ht.coupling}(ii), the random array $\hat X^{(\ell)}$ has the same law as an i.i.d.\ family of block variables
\[
m^{-\ell/\gamma}\sum_{i\in I}X_i^\beta,
\]
and the random array $\hat Y$ has the same law as an i.i.d.\ family of standard positive $\gamma$-stable variables. Therefore, after integrating the conditional graph law over $U$, we obtain exactly the annealed law obtained by first sampling the weights and then sampling the edges and self-loops conditionally independently given these weights.
\end{proof}



\end{document}